\documentclass[bj,noinfoline]{imsart}

\usepackage{amsmath, amsthm, amssymb, hyperref, color, mathtools}
\usepackage[linesnumbered,ruled]{algorithm2e}
\usepackage{graphicx}
\usepackage[all]{xypic}
\usepackage{verbatim}
\usepackage[dvipsnames]{xcolor}
\usepackage[final]{pdfpages}
\usepackage{makecell}
\usepackage{array}
\newcolumntype{?}{!{\vrule width 1pt}}
\newif\ifincludeprevious

\includeprevioustrue  

\arxiv{arXiv:1903.06110}

\startlocaldefs

\newtheorem{theorem}{Theorem}
\newtheorem{proposition}[theorem]{Proposition}
\newtheorem{lemma}[theorem]{Lemma}
\newtheorem{corollary}[theorem]{Corollary}
\newtheorem*{claim}{Claim}

\newtheorem{definition}[theorem]{Definition}
\newtheorem*{problem}{Problem}
\newtheorem{remark}[theorem]{Remark}
\newtheorem{example}[theorem]{Example}

\usepackage{tikz}
\usetikzlibrary{calc}
\usetikzlibrary{arrows}
\usetikzlibrary{decorations.pathmorphing}

\newcommand{\xx}{1}
\newcommand{\yy}{1}

\newcommand{\bt}{\tikz{\node[shape=circle,draw,inner sep=2pt] {};}}
\newcommand{\ls}[1]{{\footnotesize{#1}}}

\newcommand{\stage}[2]{\tikz[baseline=(char.base)]{
            \node[shape=circle,draw,inner sep=0.5pt,fill={#1}] (char) {${\phantom{#2}}$};}}

\newcommand{\PP}{\mathbb{P}}
\newcommand{\RR}{\mathbb{R}}

\newcommand{\ZZ}{\mathbb{Z}}
\newcommand{\NN}{\mathbb{N}}

\newcommand{\T}{\mathcal{T}}

\DeclareMathOperator{\lcm}{lcm}
\DeclareMathOperator{\im}{im}

\newenvironment{subproof}[1][\proofname]{%
  \begin{proof}[#1]%
}{%
  \end{proof}%
}

\endlocaldefs

\begin{document}

\begin{frontmatter}

\title{Discrete Statistical Models  with \\ Rational Maximum Likelihood Estimator}

\runtitle{Discrete Statistical Models with Rational MLE}

\begin{aug}

\author{\fnms{Eliana} \snm{Duarte}\thanksref{t1,t5}\ead[label=e1]{eliana.duarte@ovgu.de}},
\author{\fnms{Orlando} \snm{Marigliano}\corref{}\thanksref{t4,t2}\ead[label=e2]{orlando.marigliano@mis.mpg.de}}
\and
\author{\fnms{Bernd} \snm{Sturmfels}\thanksref{t4,t3}\ead[label=e3]{bernd@mis.mpg.de}}

\address[t4]{Max-Planck-Institut f\"ur Mathematik in den Naturwissenschaften,\\ Inselstra\ss e 22, 04103 Leipzig, Germany. \printead{e2,e3}}
\address[t1]{Fakult\"at f\"ur Mathematik, Otto-von-Guericke Universit\"at Magdeburg,\\ 39106 Magdeburg, Germany. \printead{e1}}
\runauthor{E. Duarte, O. Marigliano and B. Sturmfels}

\affiliation{Max Planck Institute for Mathematics in the Sciences}

\thankstext{t5}{Partial affiliation with Max-Planck-Institut f\"ur Mathematik in den Naturwissenschaften, Leipzig}
\thankstext{t2}{Corresponding author}
\thankstext{t3}{Partial affiliation with University of California, Berkeley}

\end{aug}

\begin{abstract}
A discrete statistical model is a subset of a probability simplex. Its maximum likelihood
estimator (MLE) is a retraction from that simplex onto the model. We characterize all models
for which this retraction is a rational function. This is a contribution via real algebraic geometry
which rests on results on Horn uniformization due to Huh and Kapranov.
We present an algorithm for constructing models with rational MLE,
and we demonstrate it on a range of instances. Our focus lies on models
familiar to statisticians, like Bayesian networks, decomposable graphical models, and staged trees.
\end{abstract}



\end{frontmatter}

\vspace{-1.5em}

\section{Introduction}

A {\em discrete statistical model} \  is a subset 
$\mathcal{M}$ of the open probability simplex  $\Delta_{n}$. Each point $p$ in $\Delta_n$
is a probability distribution on the finite state space $\{0,1,\ldots,n\}$,~i.e.~$p = (p_0,p_1,
\ldots,p_n),$ where the $p_i$ are positive real numbers that satisfy
$p_0+p_1 + \cdots + p_n=1$. The model $\mathcal{M}$ is the set of all distributions 
$p \in \Delta_n$ that are relevant for an application.

In data analysis we are given an empirical distribution
$u = (u_0,u_1,\ldots,u_n)$. This is the point in  $\Delta_n$
whose $i$th coordinate $u_i$ is the fraction of samples in state~$i$.
The {\em maximum likelihood estimator} (MLE) of $\mathcal{M}$ is a
function $\Phi\colon \Delta_n \rightarrow \mathcal{M}$ that
takes the empirical distribution $u$ to a distribution
 $\hat p = (\hat p_0,\hat p_1,\ldots,\hat p_n)$ that 
 best explains the given observations.
 Here ``best'' is understood in the sense of likelihood inference, so that $\hat p = \Phi(u)$ is the point in $\mathcal{M}$
that maximizes the {\em log-likelihood function}
 $ p \mapsto \sum_{i=0}^n u_i \cdot {\rm log}(p_i)$.
 For any vector $u$ in $\RR^{n+1}_{>0}$,
  we set  $\Phi(u) := \Phi(u/|u|)$ where $|u| = u_0+\cdots+u_n$.
 
Likelihood inference is consistent. This means that $\Phi(u) = u$ for $u \in \mathcal{M}$.
This follows from the fact that the log-likelihood function is 
 strictly concave on $\Delta_n$ and its unique maximizer is $p = u$.
Hence, the MLE $\Phi$ is a retraction from the simplex onto the model. 
 
 This point is fundamental for two fields at the crossroads of mathematics
 and data science. {\em Information Geometry} \cite{Ay} views
 the MLE as the nearest point map of a Riemannian metric on $\Delta_n$,
 given by the Kullback-Leibler divergence of probability distributions.
{\em Algebraic Statistics} \cite{DSS, sullivant2019} is concerned with models $\mathcal{M}$
whose MLE $\Phi$ is an algebraic function of $u$. This happens when
the constraints that define $\mathcal{M}$ are given in
terms of polynomials in $p$.
In this article we address a question that is fundamental for both  fields:
\\ {\em For which models $\mathcal{M}$ is the MLE $\,\Phi$ a
rational function in the empirical distribution  $u$?}  

The most basic example where the MLE is rational 
is the independence model for two binary random variables $(n=3)$.
Here $\mathcal{M}$ is a surface in the tetrahedron $\Delta_3$. That surface is a
familiar picture that serves as a point of entry for both
Information Geometry and Algebraic Statistics.
Points in $\mathcal{M}$  are positive rank one $2\times 2$ matrices 
\begin{small} $\begin{bmatrix} p_0 & p_1 \\
	p_2 & p_3 \end{bmatrix}$\end{small} whose entries sum to one.
	The data takes the form of
		a nonnegative integer $2 \times 2$ matrix $u$ of counts of observed frequencies. Hence
	$ \,|u| = u_0{+}u_1{+}u_2{+}u_3$ is the sample size,  and $u/|u|$ is
	the empirical distribution. The MLE $\hat p = \Phi(u)$ 
is evaluated by multiplying the row and column~sums of $u$:
$$ \begin{matrix}
\hat p_0 = \frac{(u_0 {+} u_1)(u_0{+}u_2)}{|u|^2}  , \,\,\,
\hat p_1 = \frac{(u_0 {+} u_1)(u_1{+}u_3)}{|u|^2}  , \, \\
\hat p_2 = \frac{(u_2 {+} u_3)(u_0{+}u_2)}{|u|^2}  , \,\,\,
\hat p_3 = \frac{(u_2 {+} u_3)(u_1{+}u_3)}{|u|^2} .
\end{matrix}
$$
These four expressions are rational, homogeneous of degree zero,
and their sum is equal to~$1$.
See \cite[Example 2]{huh14}
for a discussion of these formulas from our present perspective.

The surface $\mathcal{M}$ belongs to the class of graphical models \cite{lauritzen}.
Fix an undirected graph $G$ whose nodes represent random variables
with finitely many states.
The undirected graphical model $\mathcal{M}_G$ is a subset 
of $\Delta_n$, where $n{+}1$ is the number of states in the joint distribution.
The graphical model $\mathcal{M}_G$ is \emph{decomposable} if and only 
if the graph $G$ is chordal. 
Each coordinate $\hat p_i$ of its MLE is 
an alternating product of linear forms
given by maximal cliques and minimal separators of~$G$.
A similar formula exists for directed graphical models,
which are  also known as Bayesian networks.

In both cases, the coordinates of the MLE are not only rational functions, but even alternating products
of linear forms in $u = (u_0,u_1,\ldots,u_n)$.
This is no coincidence. Huh \cite{huh14} proved that
if $\Phi$ is a rational function then each of its coordinates is an
alternating product of linear forms, with numerator and denominator
of the same degree. Huh further showed that this alternating product must take a very
specific shape. That shape was discovered by Kapranov~\cite{Kap91},
who named it the {\em Horn uniformization}.
The results by Kapranov and Huh are valid for arbitrary
complex algebraic varieties. They make no reference
to a context where  the coordinates are real, positive, and add up to~$1$.

The present paper makes the
 leap from complex varieties back to statistical models. Building on the remarkable constructions by Kapranov and Huh,  we here work in the setting of
 real algebraic geometry that is required for statistical applications.
Our main result (Theorem~\ref{thm:main}) characterizes all models $\mathcal{M}$ in
$\Delta_n$ whose MLE is a rational function. It is stated in Section~2 and all its ingredients are presented in a self-contained manner.

In Section~3 we examine models with rational MLE that are familiar to
statisticians, such as decomposable graphical models and
Bayesian networks. Our focus lies on {\em staged tree models},
a far-reaching generalization of discrete Bayesian networks, described
in the book by Collazo, G\"orgen and Smith \cite{CGS}.
 We explain how our main result applies to these models.
 The proof of Theorem~\ref{thm:main} is presented in Section~4.
 This is the technical heart of our paper, building on the
 likelihood geometry of \cite[\S 3]{huh2014likelihood}.
 We also discuss the connection to toric geometry and geometric modeling
 developed by Clarke and Cox \cite{clarke2018moment}.
In Section~5 we present our algorithm
for constructing models with rational MLE, and we discuss
its implementation and some experiments.
The input is an integer matrix representing a toric variety, and the output is
a list of models derived from that matrix. 
Our results suggest that only a very small fraction of Huh's varieties in \cite{huh14}
are statistical models.

\section{How to be Rational}

Let $\mathcal{M}$ be a discrete statistical model in the open simplex
$\Delta_n$ that has a well-defined maximum likelihood estimator
$\Phi : \Delta_n \rightarrow \mathcal{M}$. We also write
$\Phi : \RR^{n+1}_{> 0} \rightarrow \mathcal{M}$ for the induced
map $u \mapsto \Phi(u/|u|)$ on positive vectors.
If the $n+1$ coordinates of $\Phi$ are rational functions in $u$,
 then we say that $\mathcal{M}$ 
{\em has rational MLE}.
The following is our main result.

\begin{theorem} \label{thm:main}
The following are equivalent for the statistical model
$\mathcal{M} $ with MLE~$\Phi$:
\begin{itemize}
\item[(1)] The model $\mathcal{M}$ has {\bf rational MLE}. 
\item[(2)] There exists a {\bf Horn pair} $(H,\lambda)$ such that $\mathcal{M}$ is the image of the Horn map
$$\varphi_{(H,\lambda)} : \RR^{n+1}_{>0} \to \RR^{n+1}_{>0}.$$
\item[(3)] There exists a {\bf discriminantal triple}
$(A,\Delta,{\bf m})$ such that $\mathcal{M}$ is the image 
under the  monomial map $\phi_{(\Delta,{\bf m})}$ of precisely
one orthant (\ref{eq:orthantdef}) of the dual toric variety~$Y_A^*$. 
\end{itemize}
The MLE of the model satisfies the following relation on the open orthant $\RR^{n+1}_{>0}$\rm{:}\begin{equation}
\label{eq:threemaps}
\Phi \,= \, \varphi_{(H,\lambda)} \,= \, \phi_{(\Delta,\bf m)} \circ H.
\end{equation}
\end{theorem}

This theorem matters for statistics because it reveals
when a model has an MLE of the simplest possible
closed form. Property (2) says that the polynomials
appearing in the numerators and denominators of the rational
formulas must factor into linear forms with positive coefficients.
Property (3) offers a recipe, based on toric geometry,
for explicitly constructing such models.
The advance over \cite{huh14} is that
Theorem~\ref{thm:main} deals with positive real numbers.
It hence furnishes the definitive solution
in the case of applied interest.

 The goal of this section is to define all the terms seen in parts (2) and (3) of
Theorem~\ref{thm:main}.

\begin{example} \label{ex:smalltree} 
We first discuss Theorem~\ref{thm:main} for a simple experiment:
{\em Flip a biased coin. If it shows heads, flip it again}. 
This is the model with $n=2$ given by the tree diagram

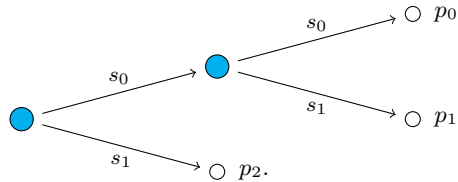
\begin{center}
\begin{tikzpicture}
\renewcommand{\xx}{1.3}
\renewcommand{\yy}{0.7}

\node (r) at (0*\xx,1*\yy) {\stage{ProcessBlue}{1}};

\node (b0) at (2*\xx,2*\yy) {\stage{ProcessBlue}{1}};
\node (b1) at (2*\xx,0*\yy) {\bt};

\node (a0) at (4*\xx,3*\yy) {\bt};
\node (a1) at (4*\xx,1*\yy) {\bt};

\draw[->] (r) -- node [above] {\ls{$s_0$}} (b0);
\draw[->] (r) -- node [below] {\ls{$s_1$}} (b1);

\draw[->] (b0) -- node [above] {\ls{$s_0$}} (a0);
\draw[->] (b0) -- node [below] {\ls{$s_1$}} (a1);

\node [right, xshift=5] at (a0) {$p_0$};
\node [right, xshift=5] at (a1) {$p_1$};
\node [right, xshift=5] at (b1) {$p_2.$};

\end{tikzpicture}
\end{center}
 The model $\mathcal{M}$ is a curve in the probability triangle $\Delta_2$.
 The tree shows its parametrization  
 $$ \qquad  \Delta_1\to \Delta_2 \,, \,\,
  (s_0,s_1)\mapsto(s_0^2,s_0s_1,s_1) 
  \qquad \hbox{where $s_0,s_1 > 0$ and $s_0+s_1=1$.} $$
  The implicit representation of the curve $\mathcal{M}$ is the equation
   $p_0p_2-(p_0+p_1)p_1=0$.
  Let $(u_0,u_1,u_2)$ be the counts from repeated experiments.
 A total of $2u_0 + 2u_1 + u_2$ coin tosses were made.
 We estimate the parameters as the empirical frequency of heads resp.~tails:
$$
	\hat s_0 \,=\, \frac{2u_0 + u_1}{2u_0 + 2u_1 + u_2} \quad \text{and}\quad 
	\hat s_1 \,=\, \frac{u_1+u_2}{2u_0 + 2u_1 + u_2}.
$$
The MLE is the retraction from the triangle $\Delta_2$ to the curve $\mathcal{M}$
given by the formula
$$ \Phi(u_0,u_1,u_2)\,\,=\,\, (\hat s_0^2, \hat s_0 \hat s_1, \hat s_1) \,\,=\,\,\,
\begin{small}
\biggl(
\frac{(2u_0 + u_1)^2}{(2u_0 {+} 2u_1 {+} u_2)^2} \,,\,
\frac{(2u_0 {+} u_1)(u_1{+}u_2)}{(2u_0 + 2u_1 + u_2)^2}\, ,\,
 \frac{u_1+u_2}{2u_0 {+} 2u_1 {+} u_2} \biggr).
 \end{small}
$$
Hence $\mathcal{M}$ has rational MLE.
We see that the Horn pair from part (2) in Theorem~\ref{thm:main}  has
$$	H \,= \,\begin{small} \begin{pmatrix}
	\phantom{-}2 & \phantom{-}1 & \phantom{-}0 \,\,\, \\
	\phantom{-}0 & \phantom{-}1 & \phantom{-}1 \, \,\,\\
	-2& -2& -1 \,\,\,
	\end{pmatrix}\end{small} \quad {\rm and} \quad \lambda \,= \,(1,1,-1). $$
We next exhibit the discriminantal triple
$(A,\Delta,\bf m)$ in
part (3) of Theorem~\ref{thm:main}.
The matrix  $A = \begin{pmatrix}1 & 1 & 1\end{pmatrix}$
gives a basis of the left kernel of $H$. The second entry is the polynomial
\begin{equation}
\label{eq:DeltaFactors} \Delta \,\, =\,\, x_3^2 - x_1^2 - x_1x_2 + x_2x_3 \,\, =\,\, 
(x_3-x_1)(x_1+x_2+x_3). \end{equation}
The third entry marks the leading term ${\bf m}= x_3^2$. These data define the monomial map
$$ \phi_{(\Delta,{\bf m})} \,\, :\,\,
(x_1,x_2,x_3) \,\mapsto \,
\biggl(\,\frac{x_1^2}{x_3^2} \,,\, \frac{x_1x_2}{x_3^2} \, \,,-\frac{x_2}{x_3} \biggr).$$

The toric variety of the matrix $A$ is the point
$Y_A  = \{(1:1:1)\}$ in $\PP^2$.  Our polynomial $\Delta$
 vanishes on the line 
 $Y_A^{*} = \{x_1+x_2+x_3 =0 \}$ that is dual to $Y_A$.
 The relevant orthant is the open line segment
 $Y^{*}_{A,\sigma} \coloneqq \{(x_1:x_2:x_3) \in Y_A^* \,: \, x_1,x_2 >0 \,\,{\rm and} \,\, x_3 <0 \}$.
Part (3) in Theorem \ref{thm:main} says that $\mathcal{M}$ is the image of
$Y^*_{A,\sigma}$ under  $\phi_{(\Delta,{\bf m})}$.
The MLE is $\Phi = \phi_{(\Delta,\bf m)} \circ H$.
\end{example}

We now come to the  definitions needed for Theorem \ref{thm:main}.
Let $H = (h_{ij})$ be an $m \times (n{+}1)$ integer matrix whose columns sum to zero,
i.e.~$\,\sum_{i=1}^m h_{ij} = 0$ for $j=0,\ldots,n$.
We call such a matrix a \emph{Horn matrix}
and denote its columns by $h_0,h_1,\ldots,h_n$.
The following alternating products of linear forms 
are rational functions of degree zero:
$$ (Hu)^{h_j} \,\,:=\,\,\prod_{i=1}^m \bigl(h_{i0} u_0 + h_{i1} u_1 + \cdots + h_{in} u_n\bigr)^{h_{ij}}
\qquad {\rm for} \,\, j=0,1,\ldots,n. $$
We use the notation $v^h \coloneqq \prod_i v_i^{h_i}$ for two vectors $v,h$ of the same size.
The Horn matrix $H$ is {\em friendly} if there 
	exists a real vector $\,\lambda = (\lambda_0,\ldots,\lambda_n)$ with $\lambda_i\neq 0$ for all $i$ such that the following identity holds in the 
	rational function field $\RR(u_0,u_1,\ldots,u_n)$:
	\begin{equation}
	\label{eq:friendly} \lambda_0  (Hu)^{h_0}  + \lambda_1 (Hu)^{h_1}+ \cdots 
	+ \lambda_n (Hu)^{h_n} \,\, = \,\, 1 .
	\end{equation}
	If this holds, then we call $(H,\lambda)$  a {\em friendly pair},
	and we consider the rational function
\begin{equation}
\label{eq:rationalmap}
		\RR^{n+1} \,\to \,\RR^{n+1} ,\,\,
		u \, \mapsto \,	\bigl(
		\lambda_0  (Hu)^{h_0}  , \,\lambda_1 (Hu)^{h_1}, \,\ldots ,\,
	 \lambda_n (Hu)^{h_n} \bigr).
\end{equation}
The friendly pair $(H,\lambda)$ is called a {\em Horn pair}
if 
the function (\ref{eq:rationalmap}) is
defined for all positive vectors, and it maps these to positive vectors.
If these conditions hold then we write 	
$\,\varphi_{(H,\lambda)} : \RR^{n+1}_{>0} \,\to \,\RR^{n+1}_{>0} \,$
for the restriction of (\ref{eq:rationalmap}) to the positive orthant. We
call $\,\varphi_{(H,\lambda)}\,$ the {\em Horn map} associated to the 
Horn pair $(H,\lambda)$.

The difference between our
 Horn pairs and the more general pairs considered by Huh in~\cite{huh14}
is
the positivity condition we just introduced, along with the ``friendliness'' condition.
These conditions guarantee that the image of the Horn map lies in the probability simplex,
which is necessary for its interpretation as a statistical model.
They also imply special properties for the Horn pair, see Propositions~\ref{prop:positivity-domino} and~\ref{cor:sigmaexists} in Section~\ref{sec:proof-main}. The examples in Section~\ref{sec:constructing-models} show that only a fraction of Huh's pairs $(H,\lambda)$ are Horn pairs.

Different Horn pairs may give rise to the same Horn map. For example, the Horn pair
\[
H' \,= \,\begin{small} \begin{pmatrix}
	\phantom{-}0 & \phantom{-}2 & \phantom{-}2 \\
	\phantom{-}2 & \phantom{-}1 & \phantom{-}0 \,\,\, \\
	\phantom{-}0 & -1 & -1 \, \,\,\\
	-2& -2& -1 \,\,\,
	\end{pmatrix}\end{small} \quad {\rm and} \quad \lambda' \,= \,\left(1,-\frac{1}{4},\frac{1}{4}\right)
\]
also gives the map
in Example~\ref{ex:smalltree}. This is because the first and third rows of $H'$ are collinear, causing the cancellation of linear factors in the Horn map. Following \cite{clarke2018moment}, a~Horn pair $(H,\lambda)$ 
is  \emph{minimal} if the matrix $H$ has no zero rows and no pair of collinear rows.

\begin{lemma}\label{minimal-reduction}
Let $(H', \lambda')$ be a Horn pair arising from the Horn pair $(H, \lambda)$ by replacing two collinear rows $r_k$ and $r_\ell$ in $H$ such that $r_\ell = \mu r_k$ with their sum $r_k + r_\ell$ and setting
\[
	\lambda'_j \,= \,\frac{\lambda_j \mu^{\mu\cdot h_{k j}}}{(1+\mu)^{(1+\mu)h_{k j}}} \quad \text{for all $\,j=0,\dotsc,n$.}
\]
Then the Horn maps $\varphi_{(H',\lambda')}$ and $\varphi_{(H,\lambda)}$ are equal.
\end{lemma}

\begin{proof}
Let $w_k$ and $w_\ell$ be the linear forms associated to the rows $r_k$ and $r_\ell$ respectively. 
Fix a column index $j$. We have $w_\ell = \mu w_k$ and
$h_{\ell j} = \mu h_{kj}$. The factors of the $j$-th coordinates of the Horn maps $\varphi_{(H,\lambda)}$ and $\varphi_{(H',\lambda')}$ that have changed after the operation are
$\lambda_j w_k^{b_{kj}} w_\ell^{b_{\ell j}} = \lambda_j\mu^{\mu\cdot h_{k j}} w_k^{(1+\mu)h_{kj}}$
for $(H,\lambda)$ and $\lambda'_j (w_k + w_\ell)^{(1+\mu)h_{k j}} = \lambda'_j (1+\mu)^{(1+\mu)h_{k j}} w_k^{(1+\mu) h_{k j}}$ for $(H',\lambda')$. 
Equating these two gives the desired formula.
\end{proof}

Every Horn map is represented by a unique minimal Horn pair.
This follows by unique factorization, see also \cite[Proposition~6.11]{clarke2018moment}.
To make a Horn pair minimal, 
while retaining the Horn map, we can use Lemma~\ref{minimal-reduction} repeatedly, deleting zero rows as they appear.

\begin{example} \label{ex:illustrate}
We illustrate the equivalence of (1) and (2) in Theorem \ref{thm:main}
for the model described in \cite[Example 3.11]{huh2014likelihood}. Here $n=3$ and $m=4$ and
the Horn matrix equals
\begin{equation}
\label{eq:matrix44} H \quad = \quad \begin{small} \begin{pmatrix}
-1 & -1 & -2 & -2 \,\, \\
 \phantom{-} 1 &   \phantom{-} 0 & \phantom{-}  3  &  \phantom{-}  2 \,\, \\
 \phantom{-}  1 &\phantom{-}  3 &  \phantom{-} 0 &  \phantom{-}  2 \,\, \\
 -1 & -2 & -1 & -2 \, \,\end{pmatrix}.\end{small}
\end{equation}
This Horn matrix is friendly because the following vector satisfies the identity (\ref{eq:friendly}):
\begin{equation}\label{eq:lambda14} \lambda \,=\, (\lambda_0,\lambda_1,\lambda_2,\lambda_3)\,=\,
\begin{small}
\biggl( \frac{2}{3} \,, \,-\frac{4}{27} \,,\, -\frac{4}{27}\,,\,\frac{1}{27} \biggr). \end{small}
\end{equation}
The pair $(H,\lambda)$ is a Horn pair, with associated Horn map
\begin{equation}
\label{eq:hornmap}
 \varphi_{(H,\lambda)} : \,\RR^{4}_{>0} \,\to \,\RR^{4}_{>0} \,,\,\,
\begin{pmatrix} u_0 \\
u_1 \\ u_2 \\ u_3 \end{pmatrix} \mapsto \begin{pmatrix} 
\frac{2 (u_0 + 3u_2 + 2u_3)(u_0+3u_1+2u_3)}{3(u_0+u_1+2u_2 + 2 u_3)(u_0+2u_1+u_2+2u_3)} \smallskip \\
\frac{4(u_0+3u_1+2u_3)^3}{27(u_0+u_1+2u_2 + 2 u_3)(u_0+2u_1+u_2+2u_3)^2} 
\smallskip \\\,
\frac{4(u_0 + 3u_2 + 2u_3)^3}{27(u_0+u_1+2u_2 + 2 u_3)^2(u_0+2u_1+u_2+2u_3)}
\smallskip \\
\frac{(u_0 + 3u_2 + 2u_3)^2(u_0+3u_1+2u_3)^2}{27(u_0+u_1+2u_2 + 2 u_3 )^2(u_0+2u_1+u_2+2u_3)^2}\,
\end{pmatrix}.
\end{equation}
Indeed, this rational function takes positive vectors to positive vectors.
The image of the map $\varphi_{(H,\lambda)}  $ is a subset $\mathcal{M}$ of the tetrahedron
$\Delta_3 = \{p \in \RR^4_{>0}:p_0+p_1+p_2 + p_3 = 1 \}$. We regard the subset
$\mathcal{M}$ as a discrete statistical model on the state space $\{0,1,2,3\}$.
The model $\mathcal{M}$ is the curve of degree $4$ inside $\Delta_3$ defined by the two quadratic equations
$$ 9p_1p_2 - 8 p_0p_3 \,=\,p_0^2 - 12p_3 \,=\,0. $$
As in \cite[Example 3.11]{huh2014likelihood},
one verifies that $\mathcal{M}$ has rational MLE, namely
$\,\Phi = \varphi_{(H,\lambda)}$.
\end{example}

We next define all the terms  used in part (3) of Theorem \ref{thm:main}.
Fix a matrix $A = (a_{ij}) \in \mathbb Z^{r\times m}$ of rank $r$ that has the vector
$(1,\ldots,1)$ in its row span. The connection to part (2) of Theorem \ref{thm:main} will be that the rows
of $A$ span the left kernel of $H$.
We identify the columns of $A$ with Laurent monomials in $r$ unknowns 
$t_1,\ldots,t_r$. The associated monomial map~is
\begin{equation}
\label{eq:monomap}
\gamma_A \,\,:\, (\RR^*)^{r} \to \mathbb \RR \mathbb{P}^{m-1}\,,\,\,\,
(t_1,\ldots,t_r) \,\mapsto \,
\biggl( \,\prod_{i=1}^r t_i^{a_{i1}}:
\,\prod_{i=1}^r t_i^{a_{i2}}:\,\, \cdots \,\, :
\,\prod_{i=1}^r t_i^{a_{im}} \biggr).
\end{equation}
Here $\RR^* = \RR \backslash \{0\}$ and $\RR \mathbb{P}^{m-1}$
denotes the real projective space of dimension $m-1$.
Let $Y_A$ be the closure of the image of $\gamma_A$. This is 
the projective toric variety given by~$A$.

Every point $x = (x_1:\cdots:x_m)$ in the dual projective space $(\RR \mathbb P^{m-1})^\vee$
corresponds to a hyperplane $H_x$ in $\RR \PP^{m-1}$. 
The \emph{dual variety} $Y_A^*$ to the toric variety $Y_A$ is the closure of 
\[
\bigl\{\,x\in (\RR \mathbb P^{m-1})^{\vee} \,\mid\, \gamma_A^{-1}(H_x\cap Y_A)
\, \text{\rm{ is singular}} \,\bigr\}.
\]
Here, the term {\em singular} means that the variety $\gamma_A^{-1}(H_x\cap Y_A)$ has a singular point in $\mathbb (\mathbb R^*)^r$. 
A general point $x$ in 
$Y_A^*$ hence corresponds to a hyperplane $H_x$ that is 
tangent to the toric variety $Y_A$ at a point $\gamma_A(t)$ with nonzero coordinates.
We identify sign vectors $\sigma \in \{-1,+1\}^m$  with orthants in $\RR^m$.
These map in a $2$-to-$1$ manner to
orthants in  $\RR \mathbb{P}^{m-1}$. If we intersect them with $Y_A^*$, 
then we get the {\em orthants} of the dual toric variety:
\begin{equation}
\label{eq:orthantdef} Y_{A,\sigma}^* \,\, = \,\, \bigl\{ \,x \in Y_A^*\, :\,  \sigma_i \cdot x_i > 0 
\,\,\hbox{for} \,\, i=1,2,\ldots,m \,\bigr\} \quad \subset \,\,\,\RR \PP^{m-1}.
\end{equation}
One of these is the distinguished orthant in Theorem~\ref{thm:main}, part (3).

\begin{example} \label{ex:fourtwo}
Fix $m=4$ and $r=2$. The following matrix has
$(1,1,1,1)$ in its row span:
\begin{equation}
\label{eq:matrix24}
 A \,=\, \begin{pmatrix} 3 & 2 & 1 & 0 \\
0 & 1 & 2 & 3 \end{pmatrix}.
\end{equation}
As in \cite[Example 3.9]{huh2014likelihood}, the toric variety of $A$
is the {\em twisted cubic curve} in  $3$-space:
$$ Y_A \,=\,\overline{\bigl\{ (t_1^3 : t_1^2t_2: t_1 t_2^2 : t_2^3) \in \RR \mathbb{P}^3
 \,: \, t_1,t_2 \in {\RR}^* \bigr\}}.$$
The dual toric variety $Y_A^*$ is a surface in
$( \RR \mathbb{P}^3)^\vee$. Its points $x$ represent planes
in $\RR \mathbb{P}^3$ that are tangent to
the curve $Y_A$. Such a tangent plane corresponds to a cubic
$\,x_1 t^3+ x_2 t^2 +x_3 t + x_4 \,$ with a double root.
Just as we recognize quadrics with a double root by the vanishing of the quadratic discriminant, a cubic with coefficients $(x_1,x_2,x_3,x_4)$ has a double root if and only if the
following discriminant vanishes:
\begin{equation}
\label{eq:disc24}
\Delta_A  \,\, = \,\,
\underline{27 x_1^2 x_4^2} - 18 x_1 x_2 x_3 x_4 + 4 x_1 x_3^3 + 4 x_2^3 x_4 - x_2^2 x_3^2.
\end{equation}
Hence, $Y_A^*$ is the surface of degree $4$ in 
$(\RR \mathbb P^{3})^\vee$ defined by $\Delta_A$.
All eight orthants $Y_{A,\sigma}^*$ are non-empty.
The coefficient vectors of the following eight cubics lie on different orthants:
$$  \begin{small}  \begin{matrix} (t+1)^2(t+3),\,
     (t+5)^2(t-1),\,
     (t-1)^2(t+3),\,
     (t+5)^2(t-8), \\
     (t-3)^2(t+1),\,
     (t-1)^2(t-3),\,
  \underline{   (t-2)^2(t+3)},\,
     (t+1)^2(t-3) .\end{matrix} \end{small}
$$
For instance,
the underlined cubic corresponds to the point  $\,x=(1,-1,-8,12 )$ in 
the orthant $Y_{A,\sigma}^*$ associated with the sign vector $\sigma =(+1,-1,-1,+1)$.
\end{example}

Let $\Delta$ be a homogeneous polynomial in $m$ variables with $n+2$ monomials
and $\mathbf m$ one of these monomials.
There is a one-to-one correspondence between such pairs $(\Delta, \mathbf m)$ and pairs
$(H,\lambda)$ where $H$ is a Horn matrix of size $m \times (n+1)$ and $\lambda$ is a coefficient vector.
Namely, for $k=0,\dotsc,n$ write $h_k^+$ resp.\ $h_k^-$ for the positive  resp.~negative part of the column vector $h_k$, so~that $h_k = h_k^+ - h_k^-$. In addition, let $\mathrm{max}_k (h_k^-)$ be the entrywise maximum of the $h^-_k$.  
We pass from pairs $(H,\lambda)$ to pairs $(\Delta,\mathbf m)$ as follows:
\begin{equation}
\label{eq:mDelta}
\mathbf{m}\, =\, x^{\mathrm{max}_k (h_k^-)}
\quad  \text{ and } \quad \Delta \,= \,\mathbf{m}\cdot 
\biggl(1-\sum_{k=0}^n \lambda_k x^{h_k} \biggr).
\end{equation}
For the converse, from pairs $(\Delta,\mathbf m)$  to pairs $(H,\lambda)$,
 we divide $\Delta$ by~$\mathbf m$ and use the same equations to determine the pair $(H,\lambda)$. Note that the polynomial $\Delta$ being homogeneous and the matrix $H$ being a Horn matrix are equivalent conditions using the equations (\ref{eq:mDelta}).
Given a pair $(\Delta,\mathbf m)$ with associated pair $(H,\lambda)$, 
we define the monomial~map
\[
\phi_{(\Delta,{\bf m})} \, : \,  (\RR^*)^m \rightarrow \RR^{n+1} , \,
	\,x \, \mapsto \,	\bigl(
		\lambda_0  x^{h_0}  , \,\lambda_1 x^{h_1}, \,\ldots ,\,
	 \lambda_n x^{h_n} \bigr).
\]	

We now present the definition that is needed for part (3) of Theorem~\ref{thm:main}.	 
\begin{definition} \label{def:dq}
A {\em discriminantal triple} $(A,\Delta,{\bf m})$ consists of
\begin{enumerate}
\item an $r \times m$ integer matrix $A$ of rank $r$ having $(1,1,\ldots,1)$ in its row span,
\item an $A$-homogeneous polynomial $\Delta$  that vanishes
on the dual toric variety $Y_A^*$,
\item a distinguished term ${\bf m}$ among those that occur in 
the polynomial $\Delta$,
\end{enumerate}
such that the pair $(H,\lambda)$ associated to $(\Delta,\mathbf m)$ is a Horn pair. Here,
the polynomial
$\Delta$ being {\em $A$-homogeneous} means that 
 $Av = Aw$ for any two exponent vectors $v$ and $w$ of $\Delta$. \end{definition}

All definitions are now complete.
We illustrate Definition~\ref{def:dq} for our running example:

\begin{example}
Let $A$ be the $2 \times 4$ matrix in (\ref{eq:matrix24}),
$\Delta = \Delta_A$ its discriminant in (\ref{eq:disc24}), and
${\bf m} = 27x_1^2x_4^2$  the special term.
Then $(A,\Delta,{\bf m})$ is a discriminantal triple with
associated sign vector $\sigma=(+1,-1,-1,+1)$.
The orthant $Y_{A,\sigma}^*$, highlighted
in Example \ref{ex:fourtwo}, is a semialgebraic surface
in $Y_A^* \subset \RR \PP^3$. This surface is mapped into 
the tetrahedron $\Delta_3$~by
\begin{equation}
\label{eq:othermap}  \phi_{(\Delta,{\bf m})} \,: \, (x_1,x_2,x_3,x_4)\,\mapsto \,
\biggl( \frac{2}{3} \frac{x_2x_3}{x_1 x_4},
-\frac{4}{27} \frac{x_3^3}{x_1 x_4^2},
-\frac{4}{27} \frac{x_2^3}{x_1^2 x_4},
\frac{1}{27} \frac{x_2^2 x_3^2}{x_1^2 x_4^2}
\biggr).
\end{equation}
The image of this map is a curve in $\Delta_3$, namely
the model $\mathcal{M}$ in Example~\ref{ex:illustrate}.
We verify (\ref{eq:threemaps}) by comparing 
(\ref{eq:hornmap}) with (\ref{eq:othermap}).
The former is obtained from the latter by setting $x=Hu$.
\end{example}

\section{Staged Trees}

We consider contingency tables $u = (u_{i_1 i_2 \cdots i_m})$ of format 
$r_1 \times r_2 \times \cdots \times r_m$. Following \cite{DSS, lauritzen}, 
these represent joint distributions of discrete statistical models with
$n+1 =r_1r_2 \cdots r_m$ states. Namely, the contingency table $u$ represents the probability distribution $p \coloneqq u/|u|.$
For any subset $C \subset \{1,\ldots,m\}$, one considers 
the marginal table $u_C$ that is obtained by summing out all indices not in $C$.
The entries of the marginal table $u_C$ are sums
of entries in $u$. To obtain the entry $u_{I,C}$ of $u_C$ for any  state $I = (i_1,i_2,\ldots,i_m),$ 
we fix the indices of the states in $C$ and sum over the indices not in $C$.
For example, if $m=4$, $C=\{1,3\}$,
$I=(i,j,k,l)$, then $u_C$ is the $r_1 \times r_3$ matrix with entries
$$ u_{I,C} \,\, = \,\, 
u_{i + k + } \,\, =\,\,\,
\sum_{j=1}^{r_2} \sum_{l=1}^{r_4} u_{ i j k l} .$$
Such linear forms are the basic building blocks for
 familiar models with rational MLE.

Consider an undirected graph $G$ with vertex set $\{1,\ldots,m\}$
which is assumed to be {\em chordal}.
The associated {\em decomposable graphical model} $\mathcal{M}_G$
in $\Delta_n$ has the rational MLE 
\begin{equation}
\label{eq:mle_dg} \hat p_I \,\, = \,\, \frac{\prod_{C} u_{I,C} }{\prod_{S} u_{I,S}}, 
\end{equation}
where the product in the numerator is over all maximal cliques $C$ of $G$,
and the product in the denominator is over all separators $S$ in
a junction tree for $G$. See \cite[\S 4.4.1]{lauritzen}.
We shall regard $G$ as a directed graph, with edge
directions given by a
 perfect elimination ordering on the vertex set $\{1,\ldots,m\}$.
This turns $\mathcal{M}_G$ into a Bayesian network.
More generally, a {\em Bayesian network} $\mathcal{M}_G$ 
 is given by a directed acyclic graph $G$. We write
 ${\rm pa}(j)$ for the set of parents of the node $j$. The model $\mathcal{M}_G$
in $ \Delta_n$
has the rational MLE
 \begin{equation}
\label{eq:mle_bn} \hat p_I \,\, = \,\,
\prod_{j=1}^m \frac{u_{I,{\rm pa}(j) \cup \{j\}}}{ u_{I,{\rm pa}(j)}}.
\end{equation}
If $G$ comes from an undirected chordal graph then 
 (\ref{eq:mle_dg}) arises from
(\ref{eq:mle_bn}) by cancellations.

\begin{example}[$m=4$]\label{expl:u-plus} We revisit two examples
from on page 36 in \cite[\S 2.1]{DSS}.
The {\em star graph} $G = [14][24][34]$ is chordal. The MLE for $\mathcal{M}_G$ is the map
$\Phi$ with coordinates
$$ \hat p_{ijkl} \,\, =\,\,\, \frac{ u_{i++l} \cdot u_{+j+l} \cdot u_{++kl}}
{ u_{++++} \cdot u_{+++l}^2} \,\, =\,\,\,
\frac{u_{i+++}}{u_{++++}} \cdot
\frac{u_{+j+l}}{u_{+++l}} \cdot
\frac{u_{++kl}}{u_{+++l}} \cdot
  \frac{u_{i ++l}}{u_{i+++}}.
$$
The left expression is (\ref{eq:mle_dg}).
The right  is (\ref{eq:mle_bn})  for the directed graph
$1 \rightarrow 4$, $4 \rightarrow 2$, $4 \rightarrow 3$.

The {\em chain graph} $G = [12][23][34]$ is chordal. Its MLE is the map
$\Phi$ with coordinates
$$
\hat p_{ijkl} \,\, =\,\,\, \frac{ u_{ij++} \cdot u_{+jk+} \cdot u_{++kl}}
{ u_{+j++} \cdot u_{++k+} \cdot u_{++++}} \,\, =\,\,\, \varphi_{(H,\lambda)}(u)_{ijkl}.
$$
This is the Horn map
given by the matrix $H$ in Figure~\ref{fig:tH} and $\lambda = (1,\dotsc,1).$
\end{example}

The formulas (\ref{eq:mle_dg})  and (\ref{eq:mle_bn}) are familiar
to statisticians. Theorem~\ref{thm:main} places them into a larger context.
 However, some readers  may find
our approach too algebraic and too general. Our aim in this section is to lay out a
useful middle ground:  staged~tree models.

Staged trees were introduced by Smith and Anderson \cite{andersonSmith} as a generalization
of discrete Bayesian networks. They furnish an intuitive representation of
 many situations that the above graphs $G$ cannot capture. In spite of their wide scope,
staged tree models are appealing because of their intuitive formalism for
encoding events. For an introduction see
the textbook~\cite{CGS}.
In what follows we  study parts (1) and (2) in Theorem~\ref{thm:main} for staged~trees.

 To define a {\em staged tree model}, we consider a directed rooted tree $\mathcal T$ 
 with at least two edges emanating from each non-leaf vertex, a label set $S = \{s_i\mid i\in I\}$, and a labeling $\theta\colon \operatorname E(\mathcal T) \to S$ of the edges of the tree. 
Each vertex of $\mathcal T$ has a corresponding \emph{floret}, which is the multiset of edge labels emanating from it. The labeled tree $\mathcal T$ is a \emph{staged tree} if any two florets are either equal or disjoint. Two vertices in $\T$ are in the same stage if their corresponding florets are the same. From now on, $F$ denotes
the set of florets of $\T$.

\begin{definition} \label{def:stm}
Let $J$ be the set of root-to-leaf paths in the tree $\T$. We set $|J| = n+1$. For $i\in I$ and $j\in J$, let $\mu_{ij}$ 
denote the number of times edge label $s_i$ appears in the $j$-th root-to-leaf path.
The \emph{staged tree model} $\mathcal M_\T $ is the image of the parametrization
$$
	\phi_{\T} : \Theta\to \Delta_n \, , \,\,
   (s_i)_{i\in I} \mapsto (p_j)_{j\in J},
$$
where the parameter space is $\,\Theta:= \bigl\{(s_i)_{i \in I} \in (0,1)^{|I|} : \sum_{s_i\in f} s_i = 1 \text{ for all florets $f\in F$}\bigr\}$,  and $p_j = \prod_{i\in I} s_i^{\mu_{ij}}$ is the product of the edge parameters on the $j$-th root-to-leaf path.
\end{definition}
In the model  $\mathcal M_\T $, the  tree $\T$ represents 
possible sequences of events. The parameter $s_i$ associated to an edge $vv'$ is the transition
probability from $v$ to $v'$. All parameter labels in a floret sum to $1$. The fact that distinct 
nodes in $\T$ can have the same floret
of parameter labels enables staged tree models  to encode conditional independence
statements \cite{andersonSmith}. This allows us to represent any discrete
Bayesian network or decomposable model as a staged tree model. 
Our first staged tree was seen in Example \ref{ex:smalltree}.
Here is another specimen.

\begin{example}[$n=15$] \label{ex:stree}
Consider the decomposable model for binary variables
 given by the $4$-chain $G=[12][23][34]$ as in Example~\ref{expl:u-plus}.
 Figure~\ref{fig:tH} shows a realization of $\mathcal{M}_G$
 as a staged tree model $\mathcal M_\T $.
The leaves of $\T$ represent the outcome space  $\{0,1\}^4$.
Nodes with the same color have the same associated floret. 
The blank nodes all have different florets. The seven florets of $\T$ are 

\begin{small}
$$
f_1 {=}\{s_0,s_1\},   f_2 {=}\{s_2,s_3\},   f_3 {=}\{s_4,s_5\},    f_4 {=} \{s_6,s_7\},
f_5 {=}\{s_8,s_9\},   f_6 {=}\{s_{10},s_{11}\},   f_7 {=}\{s_{12},s_{13}\}.
$$
\end{small}
\end{example}

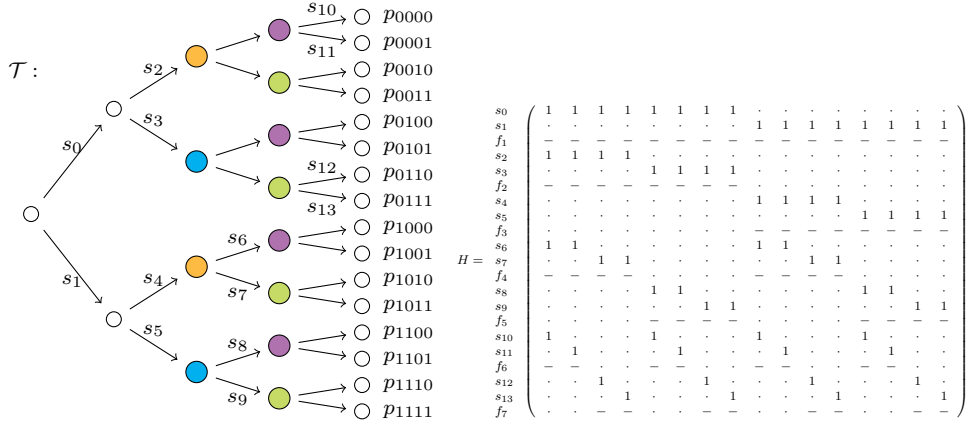
\begin{figure}[h]
\begin{center}
\begin{tikzpicture}
\renewcommand{\xx}{1.1}
\renewcommand{\yy}{0.35}
\node at (1,13*\yy) {$\T:$};

\node (r) at (1*\xx,7.5*\yy) {\bt};

\node (d1) at (2*\xx,11.5*\yy) {\bt};
\node (d2) at (2*\xx,3.5*\yy) {\bt};

\node (c1) at (3*\xx,13.5*\yy) {\stage{Dandelion}{1}};
\node (c2) at (3*\xx,9.5*\yy) {\stage{ProcessBlue}{1}};
\node (c3) at (3*\xx,5.5*\yy) {\stage{Dandelion}{1}};
\node (c4) at (3*\xx,1.5*\yy) {\stage{ProcessBlue}{1}};

\node (b1) at (4*\xx,14.5*\yy) {\stage{Orchid}{1}};
\node (b2) at (4*\xx,12.5*\yy) {\stage{SpringGreen}{1}};
\node (b3) at (4*\xx,10.5*\yy) {\stage{Orchid}{1}};
\node (b4) at (4*\xx,8.5*\yy) {\stage{SpringGreen}{1}};
\node (b5) at (4*\xx,6.5*\yy) {\stage{Orchid}{1}};
\node (b6) at (4*\xx,4.5*\yy) {\stage{SpringGreen}{1}};
\node (b7) at (4*\xx,2.5*\yy) {\stage{Orchid}{1}};
\node (b8) at (4*\xx,0.5*\yy) {\stage{SpringGreen}{1}};

\node (a1) at (5*\xx,15*\yy) {\bt};
\node (a2) at (5*\xx,14*\yy) {\bt};
\node (a3) at (5*\xx,13*\yy) {\bt};
\node (a4) at (5*\xx,12*\yy) {\bt};
\node (a5) at (5*\xx,11*\yy) {\bt};
\node (a6) at (5*\xx,10*\yy) {\bt};
\node (a7) at (5*\xx,9*\yy) {\bt};
\node (a8) at (5*\xx,8*\yy) {\bt};
\node (a9) at (5*\xx,7*\yy) {\bt};
\node (a10) at (5*\xx,6*\yy) {\bt};
\node (a11) at (5*\xx,5*\yy) {\bt};
\node (a12) at (5*\xx,4*\yy) {\bt};
\node (a13) at (5*\xx,3*\yy) {\bt};
\node (a14) at (5*\xx,2*\yy) {\bt};
\node (a15) at (5*\xx,1*\yy) {\bt};
\node (a16) at (5*\xx,0*\yy) {\bt};

\draw[->] (r) -- node [above] {\ls{$s_0$}} (d1);
\draw[->] (r) -- node [below] {\ls{$s_1$}} (d2);

\draw[->] (d1) -- node [above] {\ls{$s_2$}} (c1);
\draw[->] (d1) -- node [above] {\ls{$s_3$}} (c2);
\draw[->] (d2) -- node [above] {\ls{$s_4$}} (c3);
\draw[->] (d2) -- node [above] {\ls{$s_5$}} (c4);

\draw[->] (c1) -- node [above] {} (b1);
\draw[->] (c1) -- node [below] {} (b2);
\draw[->] (c2) -- node [above] {} (b3);
\draw[->] (c2) -- node [below] {} (b4);
\draw[->] (c3) -- node [above] {\ls{$s_6$}} (b5);
\draw[->] (c3) -- node [below] {\ls{$s_7$}} (b6);
\draw[->] (c4) -- node [above] {\ls{$s_8$}} (b7);
\draw[->] (c4) -- node [below] {\ls{$s_9$}} (b8);

\draw[->] (b1) -- node [above] {\ls{$s_{10}$}} (a1);
\draw[->] (b1) -- node [below] {\ls{$s_{11}$}} (a2);
\draw[->] (b2) -- node [above] {} (a3);
\draw[->] (b2) -- node [below] {} (a4);
\draw[->] (b3) -- node [above] {} (a5);
\draw[->] (b3) -- node [above] {} (a6);
\draw[->] (b4) -- node [above] {\ls{$s_{12}$}} (a7);
\draw[->] (b4) -- node [below] {\ls{$s_{13}$}} (a8);
\draw[->] (b5) -- node [above] {} (a9);
\draw[->] (b5) -- node [above] {} (a10);
\draw[->] (b6) -- node [above] {} (a11);
\draw[->] (b6) -- node [above] {} (a12);
\draw[->] (b7) -- node [above] {} (a13);
\draw[->] (b7) -- node [above] {} (a14);
\draw[->] (b8) -- node [above] {} (a15);
\draw[->] (b8) -- node [above] {} (a16);


\node [right, xshift=5] at (a1) {$p_{0000}$};
\node [right, xshift=5] at (a2) {$p_{0001}$};
\node [right, xshift=5] at (a3) {$p_{0010}$};
\node [right, xshift=5] at (a4) {$p_{0011}$};
\node [right, xshift=5] at (a5) {$p_{0100}$};
\node [right, xshift=5] at (a6) {$p_{0101}$};
\node [right, xshift=5] at (a7) {$p_{0110}$};
\node [right, xshift=5] at (a8) {$p_{0111}$};
\node [right, xshift=5] at (a9) {$p_{1000}$};
\node [right, xshift=5] at (a10){$p_{1001}$};
\node [right, xshift=5] at (a11){$p_{1010}$};
\node [right, xshift=5] at (a12){$p_{1011}$};
\node [right, xshift=5] at (a13){$p_{1100}$};
\node [right, xshift=5] at (a14){$p_{1101}$};
\node [right, xshift=5] at (a15){$p_{1110}$};
\node [right, xshift=5] at (a16){$p_{1111}$};

\end{tikzpicture}
\begin{tikzpicture}
\node (m) at (1*\xx,11*\yy) {
\scalebox{0.6}{$H={\begin{array}{l}
      s_0\\s_1\\f_1\\s_2\\s_3\\f_2\\s_4\\s_5\\f_3
\\s_6\\s_7\\f_4\\s_8\\s_9\\f_5\\s_{10}\\s_{11}\\f_6\\s_{12}\\s_{13}\\f_7
      \end{array} \left({
      \begin{array}{cccccccccccccccc}
      1&1&1&1&1&1&1&1&\cdot&\cdot&\cdot&\cdot&\cdot&\cdot&\cdot&\cdot\\
      \cdot&\cdot&\cdot&\cdot&\cdot&\cdot&\cdot&\cdot&1&1&1&1&1&1&1&1\\
      {-}&{-}&{-}&{-}&{-}&{-}&{-}&{-}&{-}&{-}&{-}&{-}&{-}&{-}&{-}&{-}\\
      1&1&1&1&\cdot&\cdot&\cdot&\cdot&\cdot&\cdot&\cdot&\cdot&\cdot&\cdot&\cdot&\cdot\\
      \cdot&\cdot&\cdot&\cdot&1&1&1&1&\cdot&\cdot&\cdot&\cdot&\cdot&\cdot&\cdot&\cdot\\
      {-}&{-}&{-}&{-}&{-}&{-}&{-}&{-}&\cdot&\cdot&\cdot&\cdot&\cdot&\cdot&\cdot&\cdot\\
      \cdot&\cdot&\cdot&\cdot&\cdot&\cdot&\cdot&\cdot&1&1&1&1&\cdot&\cdot&\cdot&\cdot\\
      \cdot&\cdot&\cdot&\cdot&\cdot&\cdot&\cdot&\cdot&\cdot&\cdot&\cdot&\cdot&1&1&1&1\\
      \cdot&\cdot&\cdot&\cdot&\cdot&\cdot&\cdot&\cdot&{-}&{-}&{-}&{-}&{-}&{-}&{-}&{-}\\
      1&1&\cdot&\cdot&\cdot&\cdot&\cdot&\cdot&1&1&\cdot&\cdot&\cdot&\cdot&\cdot&\cdot\\
      \cdot&\cdot&1&1&\cdot&\cdot&\cdot&\cdot&\cdot&\cdot&1&1&\cdot&\cdot&\cdot&\cdot\\
      {-}&{-}&{-}&{-}&\cdot&\cdot&\cdot&\cdot&{-}&{-}&{-}&{-}&\cdot&\cdot&\cdot&\cdot\\
      \cdot&\cdot&\cdot&\cdot&1&1&\cdot&\cdot&\cdot&\cdot&\cdot&\cdot&1&1&\cdot&\cdot\\
      \cdot&\cdot&\cdot&\cdot&\cdot&\cdot&1&1&\cdot&\cdot&\cdot&\cdot&\cdot&\cdot&1&1\\
      \cdot&\cdot&\cdot&\cdot&{-}&{-}&{-}&{-}&\cdot&\cdot&\cdot&\cdot&{-}&{-}&{-}&{-}\\
      1&\cdot&\cdot&\cdot&1&\cdot&\cdot&\cdot&1&\cdot&\cdot&\cdot&1&\cdot&\cdot&\cdot\\
      \cdot&1&\cdot&\cdot&\cdot&1&\cdot&\cdot&\cdot&1&\cdot&\cdot&\cdot&1&\cdot&\cdot\\
      {-}&{-}&\cdot&\cdot&{-}&{-}&\cdot&\cdot&{-}&{-}&\cdot&\cdot&{-}&{-}&\cdot&\cdot\\
      \cdot&\cdot&1&\cdot&\cdot&\cdot&1&\cdot&\cdot&\cdot&1&\cdot&\cdot&\cdot&1&\cdot\\
      \cdot&\cdot&\cdot&1&\cdot&\cdot&\cdot&1&\cdot&\cdot&\cdot&1&\cdot&\cdot&\cdot&1\\
      \cdot&\cdot&{-}&{-}&\cdot&\cdot&{-}&{-}&\cdot&\cdot&{-}&{-}&\cdot&\cdot&{-}&{-}\\
      \end{array}}\right)}$}};
\end{tikzpicture}
\end{center}\vspace {-0.1in}
\caption{
A staged tree $\T$ and its Horn matrix $H$ 
from Proposition~\ref{prop:mle-staged-trees}. Entries  $-$ indicate  $-1$.}\label{fig:tH}
\end{figure}

Next we show that staged tree models have rational MLE, so they 
satisfy part (1) of Theorem \ref{thm:main}.
Our formula for $\Phi$ uses the notation for
$I,J$ and $\mu_{ij}$ introduced in Definition~\ref{def:stm}.
This formula is known in the literature on
chain event graphs (see e.g.~\cite{Silander.Leong.2013}).

\begin{proposition} \label{prop:mle-staged-trees}
Let $\mathcal{M}_{\mathcal{T}}$ be a staged tree model, and let $u = (u_j)_{j\in J}$ be a
vector of counts. For $i\in I$, let $f$ be the floret containing 
the label $s_i$, and define the  estimates
\[
	\hat s_i \,\, \coloneqq \,\, \frac{\sum_j \mu_{ij}u_j}{\sum_{s_{\ell }\in f}\sum_{j} \mu_{\ell j} u_j}
\quad {\rm and} \quad
	\hat p_j \,\,\coloneqq \,\, \prod_{i \in I} (\hat s_i)^{\mu_{ij}}.
\]
The rational function $\,\Phi\,$ that sends $\,(u_j)_{j\in J}\,$ 
to $\,(\hat p_j)_{j\in J}\,$ is the MLE of the model $\mathcal M_\T$.
\end{proposition}

\begin{proof}
We prove that the likelihood function $L(p,u)$ has a unique maximum at 
$p = (\hat p_j)_{j\in J}$.
For a floret $f\in F$,  we fix the  vector of parameters $s_f = (s_i)_{s_i\in f}$,
and we define the local
 likelihood function $L_f(s_f,u) = \prod_{s_i\in f} s_i^{\alpha_i}$, where
$\alpha_i =\sum_j \mu_{ij} u_j$.  We~have
\[	L(p,u) \, = \, \prod_j p_j^{u_j}  \, = \,\prod_j \prod_i s_i ^ {u_j\mu_{ij}}\, 
= \,\prod_i s_i^{\alpha_i}  \, = \, \prod_{f\in F} L_f(s_f,u). \]
	Since the $L_f$ depend on disjoint sets of unknowns, maximizing 
	$L$ is achieved by maximizing the factors $L_f$ separately.
	But $L_f$ is the likelihood function of the full model $\Delta_{|f|-1}$, 
given the data vector $(\alpha_i)_{s_i\in f}$. The MLE of that model is
$\hat s_i = \alpha_i / \sum_{s_\ell \in f} \alpha_\ell$, where $s_i \in f$.
We conclude that $\,\mathrm{argmax}_{s_f}\bigl(L_f(s_f,u)\bigr) = (\hat s_i)_{s_i\in f}\,$ 
and  $\,\mathrm{argmax}_p \bigl( L(p,u) \bigr) = (\hat p_j)_{j \in J}$.
\end{proof}

\begin{remark}\label{elementary}
Here is a method for evaluating the MLE  in Proposition~\ref{prop:mle-staged-trees}.
Let $[v]\subset J$ be the set of root-to-leaf paths through a node $v$
in the tree $\T$	and define $u_{[v]}=\sum_{j\in [v]}u_j$.
	The ratio $\frac{u_{[v']}}{u_{[v]}}$ is  the empirical transition probability from
	$v$ to $v'$ given arrival at $v$.
	To obtain $\hat s_i$ we first compute the quotients $\frac{u_{[v']}}{u_{[v]}}$ for all edges $vv'$ with parameter
	label $s_i$. We aggregate them by adding
	their numerators and denominators separately.
	 This gives
	$\, \hat s_i=	(\sum u_{[v']})/(\sum u_{[v]})$,
	where both sums range over all edges $vv'$ with parameter
	label $s_i$.
\end{remark}

 Proposition~\ref{prop:mle-staged-trees} yields an explicit
description of the Horn pair $(H,\lambda)$ associated to $\mathcal{M}_{\T}$.

\begin{corollary} \label{cor:horn-staged-trees}
Fix a staged tree model $\mathcal{M}_{\T}$ as above.
Let      $H$ be the $(|I|+|F|)\times |J|$ matrix whose rows are indexed
     by the set $I\sqcup F$ and entries are given by
      \begin{align*}
     &h_{ij}\,=\, \mu_{ij} \text{ for } i \in I, \text{ and } \\
     &h_{fj}\,= \, -\sum_{s_{\ell} \in f}\mu_{\ell j} \text{ for }f \in F.
     \end{align*}
     Define $\lambda\in \{-1,+1\}^{|J|}$ by 
     $\lambda_j= (-1)^{\sum_{f}h_{fj}}$. Then $(H,\lambda)$ is a Horn pair for 
     $\mathcal{M}_{\T}$.
\end{corollary}

Given a staged tree $\mathcal T$, we call the matrix $H$ 
in Corollary~\ref{cor:horn-staged-trees} the \emph{Horn matrix} of~$\mathcal T$.

\begin{remark}
In Corollary~\ref{cor:horn-staged-trees}, for a floret $f$,
   let $H_{f}$ be the submatrix of $H$  with row~indices
 $\{i:s_i\in f\}\cup\{f\}$. Then $H$ is the vertical concatenation
of the matrices $H_f$ for~$f\in F$.
\end{remark}

\begin{example}\label{ex:minimalH}
For the tree $\T$ in Example~\ref{ex:stree}, the  Horn matrix 
$H$ of $\mathcal M_\mathcal T$ is given in Figure~\ref{fig:tH}.
The vector $\lambda$ of the Horn pair $(H,\lambda)$
is the vector of ones $(1,\ldots,1)\in\mathbb{R}^{16}$. 
The rows of $H$ are indexed by the florets and labels
$$(s_0,s_1,f_1,s_2,s_3,f_2,s_4,s_5,f_3
,s_6,s_7,f_4,s_8,s_9,f_5,s_{10},s_{11}, f_6,s_{12},s_{13},f_7).$$ 
Note that $(H,\lambda)$ is not minimal. 
Following the recipe in Lemma~\ref{minimal-reduction},
we can delete the rows ${s_0,s_1,f_2,f_3}$
of the matrix $H$ by summing the pairs $(s_0,f_2)$ and $(s_1,f_3)$ and deleting zero rows. The result is the minimal Horn pair $(H', \lambda')$, where $\lambda' = (-1,\dotsc, -1)$.
\end{example}

Two staged trees $\T$ and $\T'$ are called 
\emph{statistically equivalent} in  \cite{gorgenSmith}
if there exists a bijection between the sets of root-to-leaf paths of $\T$ and
$\T'$ such that, after applying this bijection, $\mathcal{M}_{\T}=\mathcal{M}_{\T'}$ in the open simplex $\Delta_n$. A staged tree model may have different but statistically equivalent tree
representations. In \cite[Theorem~1]{gorgenSmith}, it is shown that statistical equivalence
of staged trees can be determined by a sequence of operations on the trees, named \emph{swap} and \emph{resize}. One of the advantages
of describing a staged tree model via its Horn pair is that it gives a new criterion to decide whether two staged trees are statistically equivalent. This is simpler to implement than the criterion given in~\cite{gorgenSmith}.  

\begin{corollary}
Two staged trees are statistically equivalent if and only if 
their associated Horn pairs reduce to the same minimal Horn pair.
\end{corollary}

One natural operation on a staged tree $\T$ is identifying two florets of the same size.
This gives a new staged tree $\T'$ whose Horn matrix is 
easy to get from that of~$\T$. 

\begin{corollary} \label{cor:staging}
	Let $\T'$ be a staged tree arising from $\T$ by identifying two florets $f$ and~$f'$, 
	say by the bijection $(-)'\colon f\to f'$. 
	The Horn  matrix $H'$ of $\mathcal{M}_{\T'}$ arises from 
	the Horn  matrix $H$ of $\mathcal{M}_{\T}$ by replacing the blocks 
	$H_f$ and $H_{f'}$ in $H$ by the block  $H'_f$ defined by 
	\begin{align*}
	                            h'_{ij}&= h_{ij} +h_{i'j} \;\;\text{ for } s_i \in f,\\
	                            h'_{fj}&= h_{fj}+h_{f'j}.
	\end{align*}

\end{corollary}
\begin{proof}
	This follows from the definition of the Horn matrices for $\mathcal{M}_{\T}$ and
	 $\mathcal{M}_{\T'}$.
\end{proof}

\begin{example} Let $\T'$ be the tree obtained from Example~\ref{ex:stree} by
identifying florets $f_4$ and $f_5$ in $\mathcal{T}$. Then  $\mathcal{M}_{\T'}$ is the
   independence model of two random variables with four states.
\end{example}

\smallskip
Now we turn to part (3) of Theorem~\ref{thm:main}. We 
describe the triple $(A,\Delta,{\bf m})$
for a staged tree model $\mathcal{M}_\mathcal{T}$. The pair
 $(H,\lambda)$ was given in Corollary~\ref{cor:horn-staged-trees}.
 Let $A$ be any matrix whose rows span the
left kernel of $H$, set $m=|I|+|F|$, and write $s$ for the $m$-tuple of parameters
$(s_i,s_f)_{i\in I, f\in F}$. From the Horn matrix in Corollary~\ref{cor:horn-staged-trees} we see that
\[\Delta={\bf m}\cdot \left(1-\sum_j (-1)^{\epsilon_j}
\prod_i \left(\frac{s_i}{s_f}\right)^{\mu_{ij}}
\right), \]
where $f$ depends on $i$, $\,{\bf m}=\lcm(\prod_i s_f^{\mu_{ij}}:f\in F)\,$
and $\,\epsilon_j={\sum_i \mu_{ij}}$.
The sign vector $\sigma$ for the triple $(A,\Delta,\mathbf m)$ is given by
$\sigma_i=+1$ for $i\in I$ 
and $\sigma_f=-1$ for $f\in F$.
Then $Y_{A,\sigma}^{*}$  gets mapped to $\mathcal{M}_{\T}$ via $\phi_{(\Delta,{\bf m})}$. Moreover, the
map $\phi_{\T}$ from Definition~\ref{def:stm} factors through $\phi_{(\Delta,{\bf m})}$. Indeed, if we define
 $\iota: \Theta \to Y_{A,\sigma}^{*}$ by $(s_i)_{i\in I}\mapsto (s_i,-1)_{i \in I, f\in F}$, then $\phi_{\T}=\phi_{(\Delta,{\bf m})}\circ \iota$.
The following derivation is an extension of that in
\cite[Example 3.13]{huh2014likelihood}.

\begin{example} \label{ex:4chainhorn}
Let $\mathcal{M}_{\T}$ be the $4$-chain model in Example~\ref{ex:stree}. Here the
discriminant is
$$
\begin{footnotesize}
\begin{matrix}
\Delta \,=\,
  f_1 f_2 f_3 f_4 f_5 f_6 f_7 \!\!\!\!\! & \!
  - \,s_0 s_2 s_6 s_{10} f_3 f_5 f_7 - s_0 s_2 s_6 s_{11} f_3 f_5 f_7
 - s_0 s_2 s_7 s_{12} f_3 f_5 f_6 - s_0 s_2 s_7 s_{13} f_3 f_5 f_6 \\ & -\, s_0 s_3 s_8 s_{10} f_3 f_4 f_7
 - s_0 s_3 s_8 s_{11} f_3 f_4 f_7 - s_0 s_3 s_9 s_{12} f_3 f_4 f_6 - s_0 s_3 s_9 s_{13} f_3 f_4 f_6 \\ &
 - \,s_1 s_4 s_6 s_{10} f_2 f_5 f_7 - s_1 s_4 s_6 s_{11} f_2 f_5 f_7 - s_1 s_4 s_7 s_{12} f_2 f_5 f_6
 - s_1 s_4 s_7 s_{13} f_2 f_5 f_6 \\ & \,- \,s_1 s_5 s_8 s_{10} f_2 f_4 f_7 - s_1 s_5 s_8 s_{11} f_2 f_4 f_7
 - s_1 s_5 s_9 s_{12} f_2 f_4 f_6 - s_1 s_5 s_9 s_{13} f_2 f_4 f_6.
\end{matrix}
\end{footnotesize}
$$
Our notation for the parameters matches the row labels of the Horn matrix $H$ in
Figure~\ref{fig:tH}. This polynomial of degree $7$ is irreducible, so it equals the $A$-discriminant:
$\,\Delta = \Delta_A$.
The underlying matrix $A$ has format $13 \times 21$, and we represent it by its associated
toric ideal
$$
\begin{footnotesize}
\begin{matrix}
I_A \,=&\!\!\!\! \bigl\langle\,
s_{10} - s_{11}\,, \,\,
s_1 s_5 f_2 - s_0 s_3 f_3\,, \,\,s_1 s_4 f_2 - s_0 s_2 f_3\,,\,\,
s_5 s_9 f_4 - s_4 s_7 f_5\,,  \,\,s_3 s_9 f_4 - s_2 s_7 f_5, \\
&s_{12} - s_{13}, \,
s_5 s_8 f_4 - s_4 s_6 f_5, s_3 s_8 f_4 - s_2 s_6 f_5, \,
s_9 s_{13} f_6 - s_8 s_{11} f_7,
\, s_7 s_{13} f_6 - s_6 s_{11} f_7,\\ & \! \!
s_0 s_2 s_6 s_{11} - f_1 f_2 f_4 f_6,
s_0 s_2 s_7 s_{13} - f_1 f_2 f_4 f_7,
s_0 s_3 s_8 s_{11} - f_1 f_2 f_5 f_6, 
s_0 s_3 s_9 s_{13} - f_1 f_2 f_5 f_7, \\ &
s_1 s_4 s_6 s_{11} - f_1 f_3 f_4 f_6, 
s_1 s_4 s_7 s_{13} - f_1 f_3 f_4 f_7,
s_1 s_5 s_9 s_{13} - f_1 f_3 f_5 f_7, 
s_1 s_5 s_8 s_{11} - f_1 f_3 f_5 f_6 
\bigr\rangle.
\end{matrix}
\end{footnotesize}
$$
The toric variety $Y_A = \mathcal{V}(I_A)$ has dimension $12 $ and degree $141$.
It lives in a linear space of codimension $2$ in $\PP^{20}$, where it is defined by
eight cubics and eight quartics. The dual variety $Y_A^* = \mathcal{V}(\Delta_A)$
is the above hypersurface of degree seven. We have 
$ {\bf m} = f_1 f_2 f_3 f_4 f_5 f_6 f_7$, and
$\sigma $ is the vector in $\{-1,+1\}^{21}$ that has entry $+1$ at the indices corresponding to the $s_i$ and entry $-1$ at the indices corresponding to the $f_i$. 
 \end{example}

It would be interesting to study the combinatorics of 
discriminantal triples for staged tree models. Our computations suggest
that, for many such models, the polynomial $\Delta$
is irreducible and equals the $A$-discriminant 
$\Delta_A$ of the underlying configuration~$A$. However,
this is not true for all staged trees, as seen in equation (\ref{eq:DeltaFactors})
of Example~\ref{ex:smalltree}. We close this section
with a  familiar class of models with rational MLE whose associated $\Delta$ factor.

 \begin{example}\label{ex:multinomial} The {\em multinomial distribution} encodes the
 experiment of rolling a $k$-sided die $m$ times and recording the number of times one observed the $j$-th side, for $j=1,\dotsc,k$.  The associated model $\mathcal{M}$ is the
independence model for $m$ identically distributed random variables on $k$ states. We have
 $n+1 = \binom{k+m-1}{m}$. The Horn matrix
$H$ is the $(k+1) \times (n+1)$ matrix 
whose columns are the vectors $(-m,i_1,i_2,\ldots,i_k)^T$
where $i_1,i_2,\ldots,i_k$ are nonnegative integers whose sum equals $m$.
Here,  $\, A = (1 \,\, 1  \,\cdots \, 1)$, so the $A$-discriminant
equals $\,\Delta_A = x_0+x_1+\cdots+x_k$. The following
polynomial is a multiple of $\Delta_A$:
$$ \Delta \,\,= \,\,(-x_0)^m - (x_1 + x_2 + \cdots +x_k)^m. $$
This $\Delta$, with its marked term ${\bf m} = (-x_0)^m$, encodes
the MLE for the model $\mathcal{M}$:
\[ 
\hat p_{(i_1,\dotsc,i_k)} \,\,=\,\,
\prod_{j=1}^k
\left(\frac{\sum_{|I|=m} u_{I}\cdot I_j}
{m \sum_{|I|=m} u_{I}}\right)
^ {i_j}
\]
Here, $I$ ranges over all vectors in  $\mathbb{N}^k$ that sum to $m$, and $I_j$ denotes the $j$-th entry of $I$.
\end{example}

\section{Proof of the Main Theorem}\label{sec:proof-main}

In this section we prove Theorem~\ref{thm:main}.
For a pair $(H,\lambda)$ consisting of a Horn matrix $H$ and a coefficient vector $\lambda$, 
let $\varphi$ be the rational map defined in~(\ref{eq:rationalmap}). We use $\varphi$ and $\varphi_{(H,\lambda)}$ interchangeably in this section, as well as $\phi$ and $\phi_{(\Delta,\mathbf m)}$. Recall that its $j$-th coordinate~is
\begin{equation}\label{eq:rationalmap-coord}
\varphi_j(v) \,\,= \,\,\lambda_j\, \prod_{i=1}^m \biggl(\sum_{k=0}^n h_{ik}v_k \biggr)^{h_{ij}}.
\end{equation}
For a fixed data vector $u\in \mathbb N^{n+1}$, we define the likelihood function 
for the image of $\varphi$:
\begin{equation}\label{eq:likelihood-fct}
L_u \,: \,
\mathbb R^{n+1}\to \mathbb R\,,\,\,\, v \mapsto
 \,\prod_{j=0}^{n} \varphi_j(v)^{u_j}.
\end{equation}

\begin{lemma} \label{prop:horn-map-estimates-likelihoods}
Let $H = (h_{ij})$ be a Horn matrix, $\lambda$ a vector satisfying (\ref{eq:friendly}) and
$u\in \mathbb N^{n+1}$.
Then $u$ is a critical point of its own likelihood function $L_u$. Furthermore, if $u'$ is another critical point of $L_u$, then $\varphi(u)=\varphi(u')$.
\end{lemma}

\begin{proof}
We compute the partial derivatives of $L_u$.
For $\ell = 0,\dotsc, n$ we find
\begin{align*}
\frac{\partial}{\partial v_\ell} L_u(v)
&\,\,=\,\, \sum_{j=0}^n u_j\, \frac{L_u(v)}{\varphi_j(v)}\, \frac{\partial}{\partial v_\ell} \varphi_j(v)
\\ &\,\,= \,\, \sum_{j=0}^n u_j\, \frac{L_u(v)}{\varphi_j(v)}\,
\sum_{i=1}^m h_{ij}\, \frac{\varphi_j(v)}{\sum_{k=0}^{n}h_{ik} v_k}\, h_{i\ell}
\\ & \,\,=\,\, L_u(v)\, \sum_{i=1}^m \sum_{j=0}^n \frac{u_j\, h_{ij}\, h_{i\ell}}{\sum_{k=0}^n h_{ik} v_k}
\quad = \quad  L_u(v)\, \sum_{i=1}^m \frac{h_{i\ell}\,\sum_{j=0}^n h_{ij} u_j}{\sum_{k=0}^n h_{ik} v_k}.
\end{align*}
For $v=u$, this evaluates to zero, since the sums in the fraction cancel and the $\ell$-th column of $H$ sums to zero. This shows that $u$ is a critical point.

Next, let $u'$ be another critical point of $L_u$. Using terminology from \cite[Theorem~1]{huh14}, this means that $\varphi(u')$ is a critical point of the likelihood function $L(p,u)$ of the model $\mathcal M$ defined as the image of $\varphi$. The same holds for $\varphi(u)$. By the implication (ii) to (i) in \cite[Theorem~1]{huh14}, the model $\mathcal M$ has ML degree one. This implies $\varphi(u)=\varphi(u')$.
\end{proof}

We use \cite{huh14} to explain the 
relation between models with rational MLE and Horn pairs.

\begin{proof}[Proof of Theorem~\ref{thm:main}, Equivalence of (1) and (2)]
Let $\mathcal M$ be a model with rational MLE $\Phi$. The Zariski closure of
$\mathcal M$ is a variety whose likelihood function has a unique critical point. By \cite[Theorem~1]{huh14}, there is a Horn matrix $H$ and a coefficient vector $\lambda$ such that $\varphi_{(H,\lambda)} = \Phi$. Now, the required sum-to-one and positivity conditions for $\varphi_{(H,\lambda)}$ are satisfied because they are satisfied by the MLE $\Phi$. Indeed, the MLE of any 
discrete statistical model maps positive vectors 
$u$ in $\RR^{n+1}_{> 0}$ into the simplex~$\Delta_n$.
Conversely, we claim that every Horn pair $(H,\lambda)$ specifies a nonempty model $\mathcal M$ with rational MLE. Indeed, define  $\mathcal{M}$ to be the image of $\varphi_{(H,\lambda)}$. By the defining properties of the Horn pair, we have $\mathcal M \subset \Delta_n$. Lemma \ref{prop:horn-map-estimates-likelihoods} shows that $\varphi_{(H,\lambda)}$ is the MLE of $\mathcal M$.
\end{proof}

Next, we relate Horn pairs to discriminantal triples.

\begin{proof}[Proof of Theorem~\ref{thm:main}, Equivalence of (2) and (3)]
We already exhibited a bijection between pairs $(H,\lambda)$ and pairs
$(\Delta, \mathbf m)$ given by Equation~\ref{eq:mDelta}. The matrix $A$ is the left kernel of $H$ and forms the triple $(A,\Delta,\mathbf m)$. It is a matrix of size $r\times m$ of rank $r$. When $H$ is a Horn matrix, $A$ contains $(1,\dotsc,1)$ in its row span. This implies that the polynomial $\Delta$ is homogeneous, which in turn implies that it is $A$-homogeneous by $AH=0$.

Next, we show that the pair $(H,\lambda)$ being friendly corresponds to the polynomial $\Delta$ vanishing on $Y_A^*$. This is part of the desired equivalence.
 
\begin{claim}
The pair $(H,\lambda)$ is friendly if and only if the $A$-homogeneous polynomial $\Delta$ vanishes on the dual toric variety $Y_A^*$.
\end{claim}

\begin{subproof}[Proof of Claim]
Let $(H,\lambda)$ be friendly and $A$ as above.
The Laurent polynomial $\,q := \Delta /{\mathbf m}\,$
is a rational function on $\PP^{m-1}$ that vanishes on the dual toric variety $Y_A^*$. To see this, 
consider the exponentiation map 
$\,\varphi_2\,\colon \,\mathbb P^{m-1}\to \mathbb R^{n+1},\,
x \mapsto \lambda * x^H$,
 where $*$ is the entrywise product and $x^H\coloneqq (x^{h_0},\dotsc, x^{h_{n}})$.
 Let $f = 1 - (p_0 + \cdots + p_n)$. We have $q=f\circ \varphi_2$. 
 By \cite[Theorems~1 and~2]{huh14}, the function $\varphi_2$ maps an open 
 dense subset of $Y_A^*$ dominantly to the closure $\overline{\mathcal M}$ of the image of
  $\varphi_{(H,\lambda)}$. Since $f = 0$ on $\overline{\mathcal M}$, we have
   $f\circ \varphi_2 = 0$ on an open dense subset of $Y_A^*$, 
   hence $q = 0$ on $Y_A^*$, so $\Delta = 0$ there as well. 

Conversely, let $\Delta$ vanish
on $Y_A^*$. We claim that $q(x)$ is zero for all $x=Hu$ in the image of the
 linear map $H$. We may assume $\mathbf m (x) \neq 0$. We only need to show that $x$ is in the dual toric variety $Y_A^*$, since $\Delta$ vanishes on it. So, let $x_i = \sum_{j=0}^n h_{ij}u_j$ for $i=1,\dotsc m$. We 
 claim that  $t=(1,\dotsc,1)$ is a singular point of the hypersurface
	\[
	\gamma_A^{-1}(H_x\cap Y_A) \,\,=\,\,
	 \left\{t \in \mathbb C^r \mid \sum_{i=1}^m x_i t^{a_i} = 0\right\}.
	\]
First, the point $t$ lies on that hypersurface
since the columns of $H$ sum to zero:	
	\begin{align*}
	\sum_{i=1}^m x_i \,=\, \sum_{i=1}^m \sum_{j=0}^n h_{ij} u_j 
	\,= \,\sum_{j=0}^n u_j \sum_{i=1}^m h_{ij} \,=\, 0.
	\end{align*}
For $s=1,\dotsc,r$ we have $	\frac{\partial}{\partial t_s} t^{a_i} = a_{si} t^{a_i - e_s}$,
with $e_s$ the standard basis vector of $\mathbb Z^r$, and
	\begin{align*}
	\frac{\partial}{\partial t_s}\sum_{i=1}^m x_i t^{a_i}
	\,\,= \,\,\sum_{i=1}^m \sum_{j=0}^n h_{ij} u_j a_{si} t^{a_i-e_s}
	\,\,= \,\,\sum_{j=0}^n u_j \sum_{i=1}^m a_{si}h_{ij}t^{a_i-e_s}.
	\end{align*}
	This is zero at $t=(1,\dotsc,1)$ because  $AH = 0$.
\end{subproof}

We now prove the rest of the equivalence. Let $(H,\lambda)$ be a Horn pair, let $\varphi$ be its Horn map and let $\phi$ be the associated monomial map. Let $\mathcal M$ be the statistical model with MLE $\varphi$, so $\mathcal M = \varphi(\mathbb R_{>0}^{n+1})$. We have $\varphi = \phi \circ H$.
By Proposition~\ref{cor:sigmaexists}, there exists a unique sign vector $\sigma$ such that $\operatorname {im} H|_{\mathbb R^{n+1}_{>0}} \subseteq \mathbb R^{m}_\sigma$.
From the proof of the above claim we know that $\im H \subseteq Y_A^*$. Together, we have
\[
\mathcal M = \varphi(\mathbb R^{n+1}_{>0}) = \phi(\im H|_{\mathbb R^{n+1}_{>0}}) \subseteq \phi(Y_{A,\sigma}^*).
\]
By \cite[Theorems~1 and~2]{huh14} we have
$\phi(Y_A^*)\subseteq \mathcal M'$, where $\mathcal M'$ is the real part of
$\overline{\varphi(\mathbb C^{n+1})}$. We also have $\phi(Y_{A,\sigma}^*) \subseteq \mathbb R^{n+1}_{>0}$ by definition of the orthant. Thus $\phi(Y_{A,\sigma}^*) \subseteq \mathcal M' \cap \mathbb R^{n+1}_{>0}$.
Every element in the latter set is a fixed point of the rational function $\varphi$, by a similar argument as in Lemma~\ref{prop:horn-map-estimates-likelihoods} for complex space. Hence $\mathcal M' \cap \mathbb R_{>0}^{n+1} = \mathcal M$, so $\phi(Y_{A,\sigma}^*)\subseteq \mathcal M$.

Finally, if $(A, \Delta,\mathbf m)$ is a discriminantal triple then $(H,\lambda)$ is a Horn pair by definition. This completes the proof of Theorem~\ref{thm:main}.
\end{proof}

In the next two propositions, we formulate simple criteria to decide whether the image of the map $\varphi_{(H,\lambda)}$ associated to a Horn matrix $H$ and a coefficient vector $\lambda$ is a statistical model. These are essential for constructing models with rational MLE in Algorithm~\ref{algo:toric-to-models}.

\begin{proposition}\label{prop:positivity-domino}
	Let $(H,\lambda)$ be a friendly pair. If there exists a vector $u_0\in \mathbb R^{n+1}$ such that $\varphi(u_0)>0$, then we have $\varphi(u) > 0$ for all $u$
	in $\RR_{>0}^{n+1}$ where it is defined.
\end{proposition}

\begin{proof} 
The function  $\varphi$ is homogeneous of
degree zero. It suffices to prove each coordinate of $\varphi(u)$ is a positive real number,
 for all vectors $u$ with
positive integer entries. Indeed, every positive $u$ in $\RR^{n+1}$
 can be approximated by rational vectors, which can be scaled to be integral.
The open subset $U = \varphi^{-1}(\Delta_n)$ of~$\mathbb R^{n+1}$ contains $u_0$ by our assumptions. If $U=\mathbb R^{n+1}$, then we are done. Else, $U$ has a nonempty boundary $\partial U$.
 By continuity, $\partial U\subseteq \varphi^{-1}(\partial \Delta_n)$.
 The likelihood function $L_u$  for the data vector $u$ vanishes on~$\partial U$.

We claim that $L_u$ has a critical point in $U$. The closed subset $\overline U$ is homogeneous. Seen in projective space $\mathbb P^n$, it becomes compact.
The likelihood function $L_u$ is well defined on
this compact set in $\PP^n$, since it is homogeneous of degree zero,
and $L_u$ vanishes on the boundary.
 Hence the restriction $L_u|_U$ is either identically zero or it has a critical point in $U$. But,
  since $u_0\in U$ is a point with $L_u(u_0)\neq 0$, the second statement must be true.

Pick such a critical point $u'$. Since $U$ is open in $\mathbb R^{n+1}$, the point $u'$ is also critical point of $L_u$.
By Lemma~\ref{prop:horn-map-estimates-likelihoods} and since $u'\in U$,
we have $\varphi(u)=\varphi(u') > 0$.
\end{proof}

\begin{proposition} \label{cor:sigmaexists}
Let $(H,\lambda)$ be a friendly pair,
with no zero  or collinear rows in $H$.
Then $(H,\lambda)$ is a Horn pair if and only if
for every row $r_i$ of $H$ all nonzero entries of $r_i$ have the same sign $\sigma_i$, and
 the sign vector $\sigma = (\sigma_i)$ satisfies
$\lambda_j \sigma^{h_j} > 0$ for all columns $j$.
\end{proposition}

\begin{proof}
Let $(H,\lambda)$ be a Horn pair. Let $\ell_1,\dotsc, \ell_k$ be the linear forms corresponding to the rows in $H$ that have both positive and negative entries. Since $\ell_1$ has positive and negative coefficients, there exists a positive vector $u$ such that $\ell_1(u)=0$. Since $(H,\lambda)$ is minimal, we may choose $u>0$ such that $\ell_1(u)=0$ but $\ell_{k'}(u)\neq 0$ for all $k'\neq 1.$ The form $\ell_1$ appears in the numerator of some coordinate of $\varphi$, making this coordinate zero at $u$. But this contradicts the fact that $(H,\lambda)$ is a Horn pair. Therefore we cannot have rows with both positive and negative entries. The inequalities $\lambda_j \sigma^{h_j}>0$ then follow from the definition of a Horn pair by evaluating $\varphi(u)$ for some positive vector $u$.

Conversely, if the sign vector $\sigma$ is well-defined, the inequalities $\lambda_j\sigma^{h_j}>0$ imply that $\varphi(u)>0$ for all positive $u$. Hence $(H,\lambda)$ is a Horn pair.
\end{proof}

Every model with rational MLE {\bf arises from} a toric variety $Y_A$.
In some cases, the model {\bf is itself} a toric variety $Y_C$. 
It is crucial to distinguish the two matrices $A$ and $C$.
The two toric structures are very different.
For instance, every undirected graphical model is toric \cite[Proposition 3.3.3]{DSS}.
The toric varieties $ Y_C$ among staged tree models $\mathcal{M}_\mathcal{T}$ were classified in \cite{DG}.
The $4$-chain model $\mathcal{M}_\mathcal{T} =  Y_C$ {\bf is itself} a toric variety of
dimension~$7$ in $\PP^{15}$. But it {\bf arises from}
a toric variety $Y_A$ of dimension $12$ in $\PP^{20}$, seen in 
Example~\ref{ex:4chainhorn}.

Toric models with rational MLE play an important role 
in {\em geometric modeling} \cite{clarke2018moment, garciaSottile}. 
Given a matrix $C\in \mathbb{Z}^{r\times (n+1) }$ and a vector of weights 
$w \in \RR^{n+1}_{>0}$, one considers the 
\emph{scaled projective toric variety} $Y_{C,w}$ in $\RR \mathbb{P}^{n}$.
This is defined as the closure of the
image of 
\begin{equation}
\label{eq:monomapweights}
\gamma_{C,w} \,\,:\, (\RR^*)^{r} \to \mathbb \RR \mathbb{P}^{n}\,,\,\,\,
(t_1,\ldots,t_r) \,\mapsto \,
\biggl( \,w_0\prod_{i=1}^r t_i^{c_{i0}},
\,w_1\prod_{i=1}^r t_i^{c_{i1t}},\, \ldots \, ,
\,w_n\prod_{i=1}^r t_i^{c_{in}} \biggr).
\end{equation}
The set $\mathcal{M}_{C,w}$ of positive points in $Y_{C,w}$ is  a statistical 
model  in $\Delta_{n}$.
There is a natural homeomorphism
from the toric model  $\mathcal{M}_{C,w}$
onto the polytope of $C$. This is known in geometry
as the {\em moment map}. For a reference from algebraic statistics, see~\cite[Proposition~2.1.5]{DSS}.  In geometric modeling the pair $(C,w)$ defines
\emph{toric blending functions}~\cite{toricPatches}.

It is  desirable for the toric blending functions to
 have \emph{rational linear precision} \cite{clarke2018moment, toricPatches}.
  The property is rare and it depends in a subtle way on $(C,w)$.
Garcia-Puente and Sottile \cite{garciaSottile} established the connection
to algebraic statistics. They showed that 
 rational linear precision  holds for $(C,w)$ if and only if the 
 statistical model $\mathcal{M}_{C,w}$ has rational MLE.

\begin{example}\label{ex:multinomial2} The most classical 
blending functions with rational linear precision live on
the triangle $\{x \in \RR^3_{>0}: x_1{+}x_2{+}x_3 = 1\}$. 
They  are the {\em Bernstein basis polynomials}
\begin{equation}
\label{eq:bernstein}
\frac{m!}{i!j!(m-i-j)!}x_1^i x_2^j x_3^{m-i-j} \,\, \,\text{ for}\;\;\; i,j \geq 0, \,i+j \leq m. 
\end{equation}
Here $C$ is the  $3 \times \binom{m+1}{2}$ matrix
whose columns are the vectors $(i,j,m-i-j)$. The weights are $w_{(i,j)}=\frac{m!}{i!j!(m-i-j)!}$. 
The  toric model $\mathcal{M}_{C,w}$ is the multinomial family,
where (\ref{eq:bernstein}) is the probability of observing $i$ times $1$, $j$ times $2$ and $m-i-j$ times $3$ in $m$ trials.
 This model has rational MLE, as seen in Example \ref{ex:multinomial}.
 Again, notice the distinction between the two toric varieties.
Here, $Y_A$ is a point in $\PP^m$, whereas
$ Y_C$ is a surface in $\PP^{\binom{m}{2}-1}$.
\end{example}

Clarke and Cox \cite{clarke2018moment} raise the problem of characterizing 
all pairs $(C,w)$ with rational linear precision. This was solved 
by Duarte and G\"orgen \cite{DG} for pairs arising from staged trees.
While the problem remains open in general, our theory 
in this paper offers  new tools. 
We may ask for a characterization of
discriminantal triples whose models are toric.

\section{Constructing Models with Rational MLE}\label{sec:constructing-models}

Part (3) in Theorem~\ref{thm:main} allows us to construct models with rational MLE starting from 
a matrix $A$ that defines a projective toric variety $Y_A$.
To carry out this construction effectively we propose Algorithm~\ref{algo:toric-to-models}.
In most cases, the
dual variety $Y_A^*$ is a hypersurface, and we can compute
its defining polynomial $\Delta_A$, the \emph{discriminant} \cite{gkz}.
The polynomial $\Delta$ in a discriminantal triple can be any homogeneous multiple of $\Delta_A$, 
but we just take $\Delta = \Delta_A$.
For all terms $\mathbf m$ in $\Delta_A$, we check whether 
$(A,\Delta_A,\mathbf m)$ is a discriminantal triple.
  We implemented this algorithm in {\tt Macaulay2}, and our code is available online at~\cite{github}.
  
  Lines 1 and 15 of Algorithm~\ref{algo:toric-to-models}  are computations 
  with Gr\"obner bases.  
 Executing Line 15 can be very slow. It may be 
  omitted if one is satisfied with obtaining the parametric description and MLE $\Phi^{(\ell)}$ of the model $\mathcal M_{\ell}$. For the check in Line 14, we rely on Proposition~\ref{prop:positivity-domino} for correctness. A check based on the criterion in Proposition~\ref{cor:sigmaexists} is also~possible.

\begin{algorithm}[h]\label{algo:toric-to-models}
\caption{From toric varieties to statistical models}
\SetKwInOut{Input}{Input}
\SetKwInOut{Output}{Output}
\Input{An integer matrix $A$ of size $r\times m$ with $(1,\dotsc, 1)$ in its row span}
\Output{An integer $n$ and a collection of statistical models
$\mathcal M^{(\ell)} = (\Phi^{(\ell)}, I^{(\ell)})$, \\ where 
$\Phi^{(\ell)}\colon \mathbb R^{n+1}\to\mathbb R^{n+1}$ is a rational MLE for $\mathcal M^{(\ell)}$, 
and \\
$I^{(\ell)}\subseteq \mathbb R[p_0,\dotsc, p_n]$ is the defining prime ideal of $\mathcal M^{(\ell)}$.
}
\nllabel{line:discriminant}Compute the $A$-discriminant $\Delta_A\in \ZZ[x_1,\ldots, x_m]$\;
$n\leftarrow \#\mathrm{terms}(\Delta_A)-2$\;
$\mathrm{models} \leftarrow \{\}$\;
\For{$0\leq \ell\leq n+1$}{
	$\mathbf{m}\leftarrow \mathrm{terms}(\Delta_A)_\ell$\;
	$q\leftarrow 1-{\Delta_A}/{\mathbf m}$\;
	\For{$0\leq j \leq n$}{
		$\lambda_j \leftarrow \mathrm{coefficients}(q)_j$\;
		$h_j \leftarrow \mathrm{exponent\_vectors}(q)_j$\;
		$\Phi_j^{(\ell)}\leftarrow \textbf{(}u\mapsto \lambda_j \prod_{i=1}^m (\sum_{k=0}^n h_{ik}u_k)^{h_{ij}}\textbf{)}$\;
		}
	
	$H \leftarrow (h_{ij})$\;
Choose any positive vector $v$ in $\RR^{n+1}_{>0}$\;
		\If{$\Phi_j^{(\ell)}(v)>0$ for $j=0,1,\ldots,n$}
{
			Compute the ideal $I^{(\ell)}$ of the image of $\Phi^{(\ell)}$\;
			$\mathrm{models} \leftarrow \mathrm{models}\cup \{(\Phi^{(\ell)}, I^{(\ell)})\}$\;
		}
}
\textbf{return} $\mathrm{models}$\;
\end{algorithm}

\begin{example}[$r=2,m=4$]
For distinct integers $\alpha, \beta,\gamma > 0$ with 
$\gcd(\alpha,\beta,\gamma) = 1$~let
\[
A_{\alpha,\beta,\gamma} \,=\, \begin{pmatrix}
	1& 1& 1& 1\\
	0& \alpha& \beta& \gamma\end{pmatrix}.
\]
We ran Algorithm~\ref{algo:toric-to-models} for all $613$ such matrices with $0<\alpha<\beta<\gamma\leq 17$. Line 1 computes the discriminant $\Delta_A$ of the univariate polynomial
$ f(t)= x_1  + x_2  t^\alpha + x_3  t^\beta + x_4  t^\gamma $.
The number $n+2$ of terms of these discriminants equals
$7927/613 = 12.93 $ on average. Thus a total of $7927$ candidate triples
$(A,\Delta_A,{\bf m})$  were tested in Lines 12 to 21. 
Precisely 123 of these were found to be discriminantal triples.
This is a fraction of 1.55 \%.
Hence, only 1.55 \% of the resulting complex varieties permitted by
\cite{huh14} are actually statistical models.

Here is a typical model that was discovered. Take $\alpha=1,\beta=4,\gamma = 7$.
The discriminant 
$$  \begin{matrix}  \Delta_A &=&
729 x_2^4 x_3^6-6912 x_1^3 x_3^7-8748 x_2^5 x_3^4 x_4+84672 x_1^3 x_2 x_3^5 x_4+34992 x_2^6 x_3^2 x_4^2 \\ & & -351918 x_1^3 x_2^2 x_3^3 x_4^2
-46656 x_2^7 x_4^3+518616 x_1^3 x_2^3 x_3 x_4^3 \,\,\underline{- \,823543 x_1^6 x_4^4}
\end{matrix} 
$$
has $9$ terms, so $n=7$. The term ${\bf m}$ is underlined. The
associated model is a curve of degree ten in $\Delta_7$. Its prime
ideal $I^{(\ell)}$ is generated by $18$ quadrics. Among them are
$15$ binomials that define a toric surface of degree six:
$49 p_1 p_2 - 48 p_0 p_3,3 p_0 p_4 - p_2^2,  \ldots,
          361 p_3 p_7 - 128 p_5^2$.
Inside that surface, our curve is cut out by three  quadrics, like
 $ \,26068  p_2^2 +  73728  p_0  p_5 $ $ +703836  p_0  p_6+234612  p_2  p_6+
 78204  p_4  p_6+612864  p_0  p_7+212268  p_2  p_7+78204  p_4  p_7-8379  p_7^2
$.
\end{example}

\begin{example}[$r=3,m=6$]
For any positive integers $\alpha, \beta, \gamma, \varepsilon$, we consider the matrix
$$ 	A \,\,=\,\, \begin{small} \begin{pmatrix}
		0 & \alpha & \beta & 0 & \gamma & \varepsilon \\
		0 & 0 & 0 & 1 & 1 & 1 \\
		1 & 1 & 1 & 1 & 1 & 1
	\end{pmatrix}. \end{small}	$$
The discriminant $\Delta_A$ is the {\em resultant} of two  trinomials 
$\,x_1 + x_2 t^\alpha + x_3 t^\beta \,$ and $\,	x_4 + x_5 t^\gamma + x_6 t^\varepsilon$.
We ran Algorithm~\ref{algo:toric-to-models} for all 
138 such matrices with 
$	0<\alpha<\beta\leq 17,\, 0<\gamma<\varepsilon\leq 17, \,
	\gcd(\alpha,\beta) =  \gcd(\gamma,\varepsilon) = 1$.
	The number $n+2$ of terms of these discriminants equals
2665/138 = 19.31 on average. Thus a total of 2665 candidate triples
$(A,\Delta_A,{\bf m})$  were tested in Line 13.
Precisely 93 of these are  discriminantal triples.
This is only 3.49~\%. 
\end{example}

We now shift gears by looking at polynomials
$\Delta$ that are multiples of the $A$-discriminant.

\begin{example}[$r=1,m=4$]\label{ex:mult}
We saw in Examples~\ref{ex:smalltree} and~\ref{ex:multinomial} that interesting models arise 
from the matrix $A = (1\ 1\ \cdots \ 1)$ whose toric variety is just a point. 
Any homogeneous multiple $\Delta$ of the linear form $\Delta_A = x_1 + x_2 + \cdots+ x_m$
can be the input in Line 1 of Algorithm~\ref{algo:toric-to-models}. Here, taking $\Delta = \Delta_A$ results in the  model given by the full simplex $\Delta_{m-2}$.

Let $m=4$ and abbreviate $x^a=x_{1}^{a_1} x_{2}^{a_2} x_{3}^{a_3} x_{4}^{a_4}$ 
and $|a|=a_1{+}a_2{+}a_3{+}a_4$ for $a\in \NN^4$.
We conducted experiments with two families of multiples.
The first uses binomial multipliers:
$$
\Delta \, =\,(x^{a}+x^{b})\Delta_A\quad\hbox{or}\quad \Delta = (x^{a}-x^{b})\Delta_A,
$$
where $|a|=|b| \in \{1,2,\dots,8\}$ and $\gcd(x^a,x^b)=1$.
This gives $1028$ polynomials $\Delta$. The numbers of polynomials of
 degree $2,3,4,5,6,7,8,9$ is $6,\,21,\,46,\,81,\,126,\,181,\,246,\,321$.
 For the second family we use the trinomial multiples
$$
\Delta = (x^{a}+x^{b}+x^c)\Delta_A
\quad\hbox{or}\quad \Delta = (x^{a}+x^{b}-x^c)\Delta_A,
$$
where $|a|{=}|b| {=} |c| \! \in\! \{1,2,3\}$ and $\gcd(x^a\!,x^b\!,x^c)=1$.
Each list contains $4$ quadrics, $104$ cubics and $684$ quartics.
We report our findings in a table:

\smallskip
 \begin{center}
\begin{tabular}{|c|c|c|c|}
\hline
Family & Pairs $(\Delta,\mathbf m)$ & Horn pairs & Percentage \\
\hline
$(x^{a}-x^{b})\Delta_A$ & 8212 & 12 & 0.15\%\\
$(x^{a}+x^{b})\Delta_A$ & 8218 & 0 & 0\% \\
$(x^{a}+x^{b}-x^c)\Delta_A$ & 8678 & 8 & 0.01\% \\
$(x^{a}+x^{b}+x^c)\Delta_A$ & 8968 & 0& 0\%\\
\hline
\end{tabular}
 \end{center}
 \smallskip
All $12$ Horn pairs in the first family represent the same model, up to
permuting  coordinates. All are coming from the six quadrics of the family. The model is the surface in $\Delta_4$
defined by the $2 \times 2$ minors of the matrix
$ \begin{pmatrix} p_0 & p_1 & p_2 \,\,
\\ p_0{+}p_1{+}p_2 & p_3 & p_4 \,\, \end{pmatrix} $.
This is a staged tree model similar to Example 2, but now with three choices  
at each blue node instead of two.
The eight Horn pairs in the third family represent two distinct models. Four of the eight Horn pairs represent a surface in $\Delta_5$ and the rest represent a surface in $\Delta_6.$
\end{example}

Our construction of models with rational MLE starts with families where $r$ and $m$ are fixed.
However, as the entries of the matrix $A$ go up, the number $n+1$ of states increases.
This suggests the possibility of listing all models for fixed values of $n$. Is this list finite?

\begin{problem} Suppose that $n$ is fixed. Are there only finitely many models with rational MLE
in the simplex $\Delta_n$? Can we find absolute bounds, depending only on $n$, for
the dimension, degree and number of ideal generators of the associated varieties in $\PP^n$?
\end{problem}

 Algorithm~\ref{algo:toric-to-models} is a tool for studying these questions experimentally.
At present, we do not have any clear answers, even for $n=3$, where the models
are curves in a triangle.

\subsection*{Acknowledgements}
The first author was supported by the Deutsche Forschungsgemeinschaft DFG under grant 314838170, GRK 2297 MathCoRe.

\begin{small}

\end{small}
\end{document}